\documentclass{article}  
\usepackage{graphicx}

\usepackage{latexsym}
\usepackage{amsfonts,amsthm}
\usepackage{amsmath}
\usepackage{hyperref}
\usepackage{cleveref}\crefname{subsection}{subsection}{subsections}
\usepackage{thmtools}
\usepackage{thm-restate}
\usepackage{amssymb}
\usepackage{pifont}
\usepackage{longtable}
\usepackage{soul}
\usepackage{xcolor}
\usepackage{mathtools}
\usepackage{nicefrac}
\usepackage{xfrac}
\usepackage{adjustbox}
\usepackage{url}
\usepackage{upgreek}
\usepackage{dsfont}

\usepackage{tikz}
\usetikzlibrary{positioning}
\usetikzlibrary{decorations.markings}
\tikzset{every node/.style={circle}, 
	strike through/.append style={
		decoration={markings, mark=at position 0.5 with {
				\draw[-] ++ (-2pt,-2pt) -- (2pt,2pt);}
		},postaction={decorate}}
}

\usepackage[caption=false]{subfig}
\usepackage{soul}
\usepackage[]{authblk}

\usepackage{listings}
\usepackage{color}

\definecolor{codegreen}{rgb}{0,0.6,0}
\definecolor{codegray}{rgb}{0.5,0.5,0.5}
\definecolor{codepurple}{rgb}{0.58,0,0.82}
\definecolor{backcolour}{rgb}{0.95,0.95,0.92}

\lstdefinestyle{mystyle}{
	backgroundcolor=\color{backcolour},   
	commentstyle=\color{codegreen},
	keywordstyle=\color{magenta},
	numberstyle=\tiny\color{codegray},
	stringstyle=\color{codepurple},
	basicstyle=\tiny,
	breakatwhitespace=false,         
	breaklines=true,                 
	captionpos=b,                    
	keepspaces=true,                 
	numbers=left,                    
	numbersep=5pt,                  
	showspaces=false,                
	showstringspaces=false,
	showtabs=false,                  
	tabsize=2
}

\newcommand{\pycode}[1]{\fbox{\color{red}{\texttt{#1}}}}

\def\ContinueLineNumber{\lstset{firstnumber=last}}

\newcommand{\me}{\mathrm{e}}

\newcommand{\R}[0]{\mathbb{R}}
\newcommand{\Rz}[0]{\mathbb{R^Z}}
\newcommand{\Rzl}[0]{\mathbb{R^{Z_<}}}
\newcommand{\Rzlc}[0]{\mathbb{R}_c^{Z_<}}
\newcommand{\einf}[0]{\mathfrak{e}}
\newcommand*{\underdownarrow}[1]{\ensuremath{\underset{\downarrow}{#1}}}
\newcommand{\1}{\text{\ding{172}}}

\newtheorem{theorem}{Theorem}[section]
\newtheorem{lemma}[theorem]{Lemma}
\newtheorem{proposition}[theorem]{Proposition}

\newtheorem{axiom}[theorem]{Axiom}
\newtheorem{example}[theorem]{Example}
\newtheorem{remark}[theorem]{Remark}
\newtheorem{definition}[theorem]{Definition}

\newtheorem{property}[theorem]{Property}
\newtheorem{conjecture}[theorem]{Conjecture}

\DeclarePairedDelimiter\abs{\lvert}{\rvert}%
\DeclarePairedDelimiter\norm{\lVert}{\rVert}%
\DeclareMathOperator{\dis}{\texttt{d}}
\DeclareMathOperator{\diss}{\texttt{d}_\texttt{$\psi$}}
\DeclareMathOperator{\sign}{sign}

\makeatletter
\let\oldabs\abs
\def\abs{\@ifstar{\oldabs}{\oldabs*}}
\let\oldnorm\norm
\def\norm{\@ifstar{\oldnorm}{\oldnorm*}}
\makeatother

\begin{document}

\title{%
	Na\"ive Infinitesimal Analysis\\
	\large Its Construction and Its Properties}



\author{Anggha Nugraha\thanks{\href{mailto:anggha.satya@gmail.com}{\texttt{anggha.satya@gmail.com}}} \\
	\and Maarten McKubre-Jordens
	\and Hannes Diener \\
}


\date{}

\maketitle

\begin{abstract}
This paper aims to build a new understanding of the nonstandard mathematical analysis. The main contribution of this paper is the construction of a new set of numbers, $\Rzl$, which includes infinities and infinitesimals. The construction of this new set is done na\"ively in the sense that it does not require any heavy mathematical machinery, and so it will be much less problematic in a long term. Despite its na\"ivety character, the set $\Rzl$ is still a robust and rewarding set to work in. We further develop some analysis and topological properties of it, where not only we recover most of the basic theories that we have classically, but we also introduce some new enthralling notions in them. The computability issue of this set is also explored. The works presented here can be seen as a contribution to bridge constructive analysis and nonstandard analysis, which has been extensively (and intensively) discussed in the past few years.

\end{abstract}


\section{Background and Aim}
\label{sec:BackgroundandAim}
There have been many attempts to rule out the existence of inconsistencies in mathematical and scientific theories. Since the 1930s, we have known (from G\"{o}del's results) that it is impossible to prove the consistency of any interesting system (in our case, this is a system capable of dealing with arithmetic and analysis). One of the famous examples of inconsistency is as follows. Suppose we have a function $f\left( x\right) =ax^{2}+bx+c$ and want to find its first derivative. By using Newton's `definition of the derivative':

\begin{align}
  f'\left(x\right)  & =\frac{f\left( x+h\right) -f\left( x\right) }{h} \nonumber\\
  & =\frac{a\left( x+h\right) ^{2}+b\left( x+h\right) +c-ax^{2}-bx-c}{h}\nonumber \\
  & = \frac{ax^2+2axh+2h^2+bx+bh+c-ax^2-bx-c}{h}\nonumber\\
  & = \frac{2axh+2h^2+bh}{h}\nonumber\\
  & =2ax+h+b \label{eq1.1}\\
  & =2ax+b\label{eq1.2}
\end{align}

In the example above, the inconsistency is located in treating the variable $h$ (some researchers speak of it as an infinitesimal). It is known from the definition that $h$ is a small but non-trivial neighbourhood around $x$ and, because it is used as a divisor, cannot be zero. However, the fact that it is simply omitted at the end of the process (from Equation \ref{eq1.1} to Equation \ref{eq1.2}) indicates that it was, essentially, zero after all. Hence, we have an inconsistency.

This problem of inconsistency has been `resolved' in the $19^{\textnormal{th}}$ century\footnote{If we look historically, the debates of the use of infinitesimals have a long and vivid history. Their early appearance in mathematics was from the Greek atomist philosopher Democritus (around 450 B.C.E.), only to be dispelled by Eudoxus (a mathematician around 350 B.C.E.) in what was to become ``Euclidean'' mathematics.} by the concept of limit, but its (intuitive) na\"ive use is still common nowadays, e.g. in physics \cite{susskind2014theoretical}. In spite of that, interesting and correct results are still obtained. This outlines how firmly inconsistent infinitesimal reasoning (which is a reasoning with \textit{prima facie} inconsistent infinitesimals) is entrenched in our scientific community and it means that inconsistency is something that, if unavoidable, should be handled appropriately. Actually, inconsistency would not have been such a problem if the logic used was not explosive \cite{weber2009inconsistent}. The problem is that our mathematical theory is mostly based on classical logic, which is explosive. Thus, one promising solution is to change the logic into a non-explosive one and this is the main reason for the birth of paraconsistent mathematics which uses paraconsistent logic as its base.

Recent advances in paraconsistent mathematics have been built on developments in set theory \cite{weber2010extensionality}, geometry \cite{mortensen2010inconsistent}, arithmetic \cite{mortensen1995inconsistent}, and also the elementary research at calculus \cite{mortensen1995inconsistent} and \cite{brown2004chunk}. A first thorough study to apply paraconsistent logic in real analysis was based on the early work, such as \cite{da1974theory} and \cite{mortensen1990models}. While Rosinger in \cite{rosinger:hal-00552058} and \cite{rosinger2008safe} tried to elaborate the basic structure and use of inconsistent mathematics, McKubre-Jordens and Weber in \cite{mckubre2012real} analysed an axiomatic approach to the real line using paraconsistent logic. They succeed to show that basic field and also compactness theorems hold in that approach. They can also specify where the consistency requirement is necessary. These preliminary works in \cite{mckubre2012real} and \cite{weber2017paraconsistent} show how successful a paraconsistent setting to analysis can be. On the side of the non-standard analysis, it can be seen for example in \cite{arkeryd2012nonstandard} and \cite{fletcher2017approaches} that it is still well-studied and still used in many areas.

The underlying ideas of the research described on this paper are as follows. We have two languages: $\mathfrak{L}$, the language of real numbers $\mathbb{R}$, and $\mathfrak{L^*}$, the language of the set of hyperreal numbers $\mathbb{^*R}$. The language $\mathfrak{L^*}$ is an extension of the language $\mathfrak{L}$. It can be shown that each of those two sets forms a model for the formulas in its respective language.

Speaking about the hyperreals, the basic idea of this system is to extend the set $\mathbb{R}$ to include infinitesimal and infinite numbers without changing any of the elementary axioms of algebra. Transfer principle holds an important role in the formation of the set $\mathbb{^*R}$ in showing what are still preserved in spite of this extension. However, there are some problems with the transfer principle, notably its non-computability (see \Cref{subsec:TransferPrincipleandItsProblems}). To avoid its use, one can logically think to simply collapsing the two languages into one language $\mathfrak{\widehat{L}}$ which corresponds to the set $\mathbb{\widehat{R}}$, a new set of numbers constructed by combining the two set of axioms of $\mathbb{R}$ and $\mathbb{^*R}$. This is what we do here (see \Cref{sec:TheCreationofTheNewSets}). Nevertheless, there is at least one big problem from this idea: a contradiction.

We can at present consider two possible ways of resolving this contradiction. The first way is to change the base logic into paraconsistent logic. There are many paraconsistent logics that are available at the moment. This could be a good thing, or from another perspective, be an additional difficulty as we need to choose wisely which paraconsistent logic we want to use at first, i.e. which one is the most appropriate or the best for our purpose. But then, to be able to do this, we need to know beforehand which criteria to use and this, in itself, is still an open question.

The second way we could consider is to have a subsystem in our theory. This idea arose from a specific reasoning strategy, Chunk \& Permeate, which was introduced by Brown and Priest in 2003 \cite{brown2004chunk}. Using this strategy, we divide our set $\mathbb{\widehat{R}}$ into some consistent chunks and build some permeability relations between them. This process leads us to the creation of the sets $\Rz$ and $\Rzl$. In our view, this idea is more sensible and promises to be more useful than the first. Moreover, after further analysing this idea, we produce some new interesting and useful notions that will be worth to explore even further (see Sections \ref{sec:TopologyOnRZ<}-\ref{sec:ComputabilityInRZ<}).

The aim of this paper is to build a new model of the nonstandard analysis. By having the new set produces in this paper, we would have real numbers, infinities, and infinitesimals in one set and would still be able to do our ``usual'' analysis in, and with this set. Moreover, in terms of G\"odel's second incompleteness theorem, if we can build a new structure for nonstandard mathematical analysis which is resilient to contradiction, we would open the door to having not just a sound, but a complete mathematical theory. To put it simply, like Weber said, `In light of G\"odel's result, an inconsistent foundation for mathematics is the only remaining candidate for completeness' \cite{weber2009inconsistent}.

\section{Some Preliminaries}
\label{sec:SomePreliminaries}

To understand and carry our investigation, it is essential to have an accurate grasp of the received view about formal language, the reals and hyperreals, and the paraconsistent logic. In turn, to analyse these matters, it will be useful to fix some terminologies.

\subsection{Reals, Hyperreals, and Their Respective Languages}

Formal language is built by its syntax and semantics. Here are the symbols that are used in our language:

\begin{center}
\begin{tabular}{lll}
variables & : & $a$ $b$ $c$ $\dots$ $x_1$ $x_2$ $\dots$\\
grammatical signs & : & ( ) ,\\
connectives & : & $\land$ $\lor$ $\neg$ $\rightarrow$\\
quantifiers & : & $\forall$ $\exists$\\
constant symbols & : & $1$ $-2.5$ $\pi$ $\sqrt{5}$ $\dots$\\
function symbols & : & $+$ $-$ $\sin$ $\tan$ $\dots$\\
relation symbols & : & $=$ $<$ $>$ $\leq$ $\geq$ $\dots$\\
\end{tabular}
\end{center}

Like in natural language, a sentence is built by its term. Terms and sentences are defined as usual. \Cref{ex:LanguageforZ} uses the simple language $\mathfrak{I}$ to build some statements about integer numbers, $\mathbb{Z}$.

\begin{example}
\label{ex:LanguageforZ}
In addition to our usual connectives, variables, quantifiers and grammatical symbols in $\mathbb{Z}$, $\mathfrak{I}$ also contains:
\begin{center}
\begin{tabular}{lll}
constant symbols &: & $\dots,-2,-1,0,1,2,\dots$\\
function symbols &: & $q(x)=x^2$\\
& & $\text{add}(x,y)=x+y$\\
& & $\text{mul}(x,y)=x\times y$\\
relation symbols &: &$P(x)\textnormal{ for ``}x\textnormal{ is positive''}$\\ 
& & $E(x,y)\textnormal{ for ``}x\textnormal{ and }y\textnormal{ are equal''}$
\end{tabular}
\end{center}
In this language $\mathfrak{I}$, one can translate an English statement `squaring any integer number will give a positive number' as $\forall x$ $P(s(x))$.
\end{example}

Now we can define a language for the set hyperreals. We define a language $\mathfrak{L}$ whose every sentence, if true in reals, is also true in hyperreals.

\begin{definition}[Language $\mathfrak{L}$]
\label{def:LanguageL}
The language $\mathfrak{L}$ consists of the usual defined variables, connectives, and grammatical signs in $\mathbb{R}$, and the following:
\begin{center}
\begin{tabular}{llp{7cm}}
constant symbols &: one symbol for every real number\\
function symbols &: one symbol for every real-valued function of any finite\\ 
&\hspace{0.35em} number real variables\\
relation symbols &: one symbol for every relation on real numbers of any finite\\
&\hspace{0.35em} number real variables\\
\end{tabular}
\end{center}
\end{definition}

Semantics in our language is described by its model. This model gives an interpretation of the sentences of the language, such that we may know whether they are true or false in that model.

\begin{definition}[Model of a Language]
\label{def:ModelofaLanguage}
Suppose that we have a language $\mathfrak{A}$. A model for $\mathfrak{A}$ consists of:
\begin{enumerate}
\item a set $A$ so that each constant symbol in $\mathfrak{A}$ corresponds to an element of $A$,
\item a set $F$ of functions on $A$ so that each function symbol in $\mathfrak{A}$ corresponds to a function in $F$,
\item a set $R$ of relations on $A$ so that each relation symbol in $\mathfrak{A}$ corresponds to a relation in $R$.
\end{enumerate}
\end{definition}

\begin{example}
For our language $\mathfrak{I}$ over integer numbers, its model is the set $A=\mathbb{Z}$ with several functions and relations already well-defined in $\mathbb{Z}$.
\end{example}

\begin{restatable}[Reals as a Model]{theorem}{}
\label{thm:RealsasaModel}
The real number system $\mathbb{R}$ is a model for the language $\mathfrak{L}$.
\end{restatable}

\begin{proof}
Take $A=\mathbb{R}$ and $F$ and $R$ as set of all functions and relations, respectively, which are already well-defined in $\mathbb{R}$.
\end{proof}

Then, by using the definition of a model, we defined what hyperreal number system is.

\begin{definition}
A hypereal number system is a model for the language $\mathfrak{L}$ that, in addition to all real numbers, contains infinitesimal and infinite numbers.
\end{definition}

Now suppose that $\mathbb{^*R}$ is the set of all hyperreal numbers. Our goal now is to show that $\mathbb{^*R}$ is a model for the language $\mathfrak{L}$. To show this, we need to extend the definition of relations and functions on $\mathbb{R}$ into $\mathbb{^*R}$. This extension can also be seen in \cite{goldblatt1998lectures}.

\begin{definition}[Extended Relation]
\label{def:ExtendedRelation}
Let $R$ be a $k$-variable relation on $\mathbb{R}$, i.e. for every $x_1,x_2,\dots,x_k$, $R(x_1,x_2,\dots,x_k)$ is a sentence that is either true or false. The extension of $R$ to $\mathbb{^*R}$ is denoted by $^*R$. Suppose that $\mathbf{x}_1,\mathbf{x}_2,\dots,\mathbf{x}_k$ are any hyperreal numbers whose form is $\{x_{1n}\},\{x_{2n}\},\dots,\{x_{kn}\}$, respectively. We define $^*R(\mathbf{x_1},\mathbf{x_2},\dots,\mathbf{x_k})$ as true iff
\begin{center}
$\{n|R(x_{1n},x_{2n},\dots,x_{kn})$ is true in $\mathbb{R}\}$
\end{center}
is big.\footnote{A `big set' is a set of natural numbers so large that it includes all natural numbers with the possible exception of finitely many \cite{goldblatt1998lectures}.} Otherwise, $R(\mathbf{x_1},\mathbf{x_2},\dots,\mathbf{x_k})$ is false.
\end{definition}

\begin{example}
Suppose that
\begin{center}
$\textcircled{z} = \{1,2,3,4,5,...\} $
\end{center}
By taking $k=1$ in \Cref{def:ExtendedRelation}, we are able to have a relation $I(x)=$``$x$ is an integer''. The relation $I(\textcircled{z})$ is true. This is because the set of indexes where relation $I(x)$ is tru, is a big set. Thus, we conclude that \textcircled{z} is actually a hyperinteger.
\end{example}

\begin{definition}[Extended Function]
\label{def:ExtendedFunction}
Let $f$ be an $k$-variables function on $\mathbb{R}$. The extension of $f$ to $\mathbb{^*R}$ is denoted by $^*f$. Suppose that $\mathbf{x}_1,\mathbf{x}_2,\dots,\mathbf{x}_k$ are any hyperreal numbers whose form is $\{x_{1n}\},\{x_{2n}\},\dots,\{x_{kn}\}$, respectively. We define $^*f(\mathbf{x}_1,\mathbf{x}_2,\dots,\mathbf{x}_k)$ by
\begin{center}
$^*f(\mathbf{x}_1,\mathbf{x}_2,\dots,\mathbf{x}_k)=
\{f(x_{11},x_{21},\dots,x_{k1}),f(x_{12},x_{22},\dots,x_{k2}),f(x_{13},x_{23},\dots,x_{k3}),\dots\}$
\end{center}
\end{definition}

\begin{example}
Suppose that
\begin{center}
$\textcircled{e} = \{2,4,6,8,...\} $
\end{center}
By taking $k=1$ in \Cref{def:ExtendedFunction}, we might have, for example, a well-defined hypersinus function:
\begin{center}
$\sin(\textcircled{e})=\{\sin(2),\sin(4),\sin(6),\sin(8),...\}$
\end{center} 
\end{example}


\begin{restatable}[Hyperreals as a Model]{theorem}{}
\label{thm:HyperrealsasaModel}
The set $\mathbb{^*R}$ is a model for the language $\mathfrak{L}$ that contains infinitesimals and infinities.
\end{restatable}

\begin{proof}
Take $A=\mathbb{^*R}$ in \Cref{def:ModelofaLanguage} with all of the functions $^*f$ defined in \Cref{def:ExtendedFunction} and relations $^*R$ defined in \Cref{def:ExtendedRelation}.
\end{proof}

\subsection{Transfer Principle and Its Problems}
\label{subsec:TransferPrincipleandItsProblems}

\begin{definition}[Transfer Principle]
\label{def:TransferPrinciple}
Let $S$ be a sentence in $\mathfrak{L}$. The transfer principle says that:
\begin{center}
$S$ is true in the model $\mathbb{R}$ for $\mathfrak{L}$ iff $S$ is true in the model $\mathbb{^*R}$ for $\mathfrak{L}$.
\end{center}
\end{definition}

\noindent As Goldblatt said in \cite{goldblatt1998lectures}, `The strength of nonstandard analysis lies in the ability to transfer properties between $\mathbb{R}$ and $\mathbb{^*R}$.' But, there are some serious problems with the transfer principle. Some of them are: it is non-computable in the sense of there is no good computable representation of the hyperreals to start with; it really depends intrinsically on the mathematical model or language we use; we are prone to get things wrong when not handled correctly (especially because of human error). 

Furthermore, there is an all too often overlooked, yet major deficiency with the transfer principle: it performs particularly poorly upon a rather simple ``cost-return'' analysis. Namely, on the one hand, the mathematical machinery which must be set up in advance in order to use the transfer principle is of such a considerable technical complication and strangeness with a lack of step-by-step intuitive insight that, ever since 1966, when Abraham Robinson published the first major book on the subject --- that is, for more than half a century by now --- none of the more major mathematicians ever chose to switch to the effective daily use of nonstandard analysis, except for very few among those who have dealt with time continuous stochastic processes, and decided to use the ``Loeb Integral'' introduced in 1975. On the other hand, relatively few properties of importance can ever be transferred, since they are not --- and cannot be, within usual nonstandard analysis --- formulated in terms of first order logic.

One possible solution for overcoming (some of) these problems is simply by not using the transfer principle, by throwing together the two sets $\R$ and $^*\R$, i,e. combining their languages and axioms. However, \Cref{ex:ExampleContradictionwhenJoiningLanguages} shows that if we just simply combine the two languages, it will pose one big problem of contradictions which can lead to absurdity if we use classical logic. That is the reason why we use paraconsistent logic as it is resilient against local contradiction.

\begin{example}
\label{ex:ExampleContradictionwhenJoiningLanguages}
Take the well-ordering principle for our example. This principle says that: ``every non-empty set of natural numbers contains a least element''. Call a set $S=\{x\in N^*: x \textnormal{ is infinite}\}$. Let $s$ be its least element. Note that $s$ is infinite and so $s-1$. Thus, $s-1\in S$ and it makes $s$ is not the least element. Therefore, there exists $l$ such that $l$ is the least element of $S$ and there is no $l$ such that $l$ is the east element of $S$.
\end{example}

In addition to the problems with the transfer principle, there are also some downsides of the construction of the set $\mathbb{^*R}$ itself. Indeed, the usual construction of the hyperreal set $\R$ involves an ultrafilter on $\mathbb{N}$, the existence of which is justified by appealing to the full Axiom of Choice whose validity is still a great deal to discuss \cite{tao2012acheap,kanovei2013nonstandard}. Moreover, it also relies on some heavy and non-constructive mathematical machineries such as Zorn's lemma, the Hahn-Banach theorem, Tychonoff's theorem, the Stone-Cech compactification, or the boolean prime ideal theorem. On the side of the nonstandard analysis itself, there are some critiques as can be seen in \cite{bishop1977h,earman1975infinities,edwards2007euler,gray2015real,sergeyev2015olympic,spalt2002cauchy}. Most of them are related with its non-constructivism and its difficulties to be used in class teaching. This problem can be solved by building a na\"\i ve constructive non-standard set and making sure that it is still a useful set by redefining some well-known notions in there.

\subsection{Paraconsistent Logics in Mathematics}

Generally, paraconsistent logics are logics which permit inference from inconsistent information in a non-trivial fashion \cite{priest2002paraconsistent}. Paraconsistent logics are characterized by rejecting the universal validity of the principle \textit{ex contradictione quodlibet} (ECQ) which is defined below.

\begin{definition}[ECQ Principle]
\label{def:ECQPrinciple}
The principle of explosion, ECQ, is the law which states that any statement can be proven from a contradiction.
\end{definition}

\noindent By admitting the ECQ principle in one theory, if that theory contains a single inconsistency, it becomes absurd or trivial. This is something that, in paraconsistent logics, does not follow necessarily.

Paraconsistent logicians believe that some contradictions does not necessarily make the theory absurd. It just means that one has to be very careful when doing deductions so as to avoid falling from contradiction into an absurdity. In other words, classical and paraconsistent logic treat contradiction in different ways. The former treats contradiction as a global contradiction (making the theory absurd), while the latter treats some contradictions as a local contradiction. In other words, classical logic cannot recognise if there is an interesting structure in the event of a contradiction.

\begin{definition}[Paraconsistent Logic]
\label{def:ParaconsistentLogic}
Suppose that $A$ is a logical statement. A logic is called paraconsistent logic iff
\begin{center}
$\exists A,B$ such that $A\land \lnot A\nvdash B$.
\end{center}
\end{definition}

\noindent The symbol $\Gamma \vdash A$ simply means that there exists a proof of $A$ from set of formulas $\Gamma$, in a certain logic.

There are at least two different approaches to paraconsistent logics. The first is by adding another possible value, both true and false, to classical truth values while the second one is called the \textit{relevant-approach}. The idea of the relevant-approach is simply to make sure that the conclusion of an implication must be \textbf{relevant} to its premise(s). Those two paraconsistent logics, respectively, are Priest's Paraconsistent Logic \textbf{LP{$^\supset$}} and Relevant Logic \textbf{R}.\footnote{Note that in general, relevant logic differs from paraconsistent logic. When someone claims that they use relevant logic, it implies that they use paraconsistent logic, but not vice versa. Using paraconsistent logic does not necessarily mean using relevant logic, e.g. the logic $LP^\supset$ below is not relevant logic.} More explanations on each of them can be seen, for example, in \cite{priest1979logic}, \cite{avron1991natural}, \cite{priest1991minimally}, \cite{mares2014relevance}, and \cite{dunn2002relevance}.

When we are applying paraconsistent logic to a certain theory, there will be at least two terms that we have to be aware of: \textit{inconsistency} and \textit{incoherence}. The first term, inconsistency, is applicable if there occurs a contradiction in a system. Meanwhile, the second term, incoherence, is intended for a system which proves anything (desired or not). In classical logic, there is no difference between these two terms because of the ECQ principle. Thus, if a contradiction arises inside a theory, anything that the author would like to say can be proved or inferred within that theory. This is something that likewise does not have to happen in paraconsistent logic.

In mathematical theory, foundation of mathematics is the study of the basic mathematical concepts and how they form more complex structures and concepts. This study is especially important for learning the structures that form the language of mathematics (formulas, theories, definitions, etc.), structures that often called metamathematical concepts. A philosophical dimension is hence central to this study. One of the most interesting topics in the foundation of mathematics is the foundation of real structure, or analysis.

Generally, it is known that the construction of real numbers is categorised in classical logic --- while there is an advancement in paraconsistent logic such as in \cite{batens2007frontiers}, this has not yet been extensively explored. However, it seems viable to make a further study of real structure by developing paraconsistent foundations of analysis.

\section{The Creation of The New Sets}
\label{sec:TheCreationofTheNewSets}

As noted before, the transfer principle is useful as well as fairly problematic at the same time. One way to avoid the unnecessary complications of the transfer principle is by collapsing the two languages involved into one language $\mathfrak{\widehat{L}}$. Simply collapsing the two, however, causes additional problems. One of the problems that can be expected to appear is contradiction, but we can use a paraconsistent logic to handle this when it arises.

In this section, we  construct a new number system $\mathbb{\widehat{R}}$ through its axiomatisation by `throwing' the axioms of $\R$ and $\mathbb{^*R}$ together. The set $\mathbb{\widehat{R}}$ is the number system on which the language $\mathfrak{\widehat{L}}$ will be based and it contains positive and negative infinities, and also infinitesimals.\footnote{We are not necessarily expecting the resulting system to have contradictions, but we will make sure that we maintain coherency by not allowing contradictions to become an absurdity.} It does make sense to insert infinities (and their reciprocals, infinitesimals) into $\mathbb{\widehat{R}}$ as some of the contradictions in mathematics come from their existence and also because they are still used in today's theory as can be seen in \cite{susskind2014theoretical}. 

\subsection{The New Set $\mathbb{\widehat{R}}$}

For the sake of clarity, Axioms \ref{ax:AdditivePropertyOfRHat}--\ref{ax:ArchimedeanPropertyOfRHat} give the axiomatisation of the number system $\mathbb{\widehat{R}}$.

\begin{axiom}[Additive Property of $\mathbb{\widehat{R}}$]
	\label{ax:AdditivePropertyOfRHat}
	In the set $\mathbb{\widehat{R}}$, there is an operator $+$ that satisfies:\\
	\begin{tabular}{lp{10cm}}
		A1: & For any $x,y\in\mathbb{\widehat{R}}, x+y\in\mathbb{\widehat{R}}$.\\
		A2: & For any $x,y\in\mathbb{\widehat{R}}, x+y=y+x$.\\
		A3: & For any $x,y,z\in\mathbb{\widehat{R}}, (x+y)+z=x+(y+z)$.\\
		A4: & There is $0\in\mathbb{\widehat{R}}$ such that $x+0=x$ for all $x\in\mathbb{\widehat{R}}$.\\
		A5: & For each $x\in\mathbb{\widehat{R}}$, there is $-x\in\mathbb{\widehat{R}}$ such that $x+(-x)=0$.\\
	\end{tabular}
\end{axiom}

\begin{axiom}[Multiplicative Property of $\mathbb{\widehat{R}}$]
	\label{ax:MultiplicativePropertyOfRHat}
	In the set $\mathbb{\widehat{R}}$, there is an operator $\cdot$ that satisfies:\\
	\begin{tabular}{lp{11cm}}
		M1: & For any $x,y\in\mathbb{\widehat{R}}, x\cdot y\in\mathbb{\widehat{R}}$.\\
		M2: & For any $x,y\in\mathbb{\widehat{R}}, x\cdot y=y\cdot x$.\\
		M3: & For any $x,y,z\in\mathbb{\widehat{R}}, (x\cdot y)\cdot z=x\cdot(y\cdot z)$.\\
		M4: & There is $1\in\mathbb{\widehat{R}}$ such that $(1=0)\rightarrow\bot$ and $\forall x\in\mathbb{\widehat{R}}$ $x\cdot 1=x$.\\
		M5: & For each $x\in\mathbb{\widehat{R}}$, if $(x=0)\rightarrow\bot$, then there is $y\in\mathbb{\widehat{R}}$ such that $x\cdot y=1$.\\  
	\end{tabular}
\end{axiom} 

\begin{axiom}[Distributive Property of $\mathbb{\widehat{R}}$]
	\label{ax:DistributivePropertyOfRHat}
	For all $x,y,z\in\mathbb{\widehat{R}}$, $x\cdot(y+z)=(x\cdot y)+(x\cdot z)$.
\end{axiom}

\begin{axiom}[Total Partial Order Property of $\mathbb{\widehat{R}}$]
	\label{ax:TotalOrderPropertyOfRHat}
	There is a relation $\leq$ in $\mathbb{\widehat{R}}$, such that for each $x,y,z\in\mathbb{\widehat{R}}$:\\
	\begin{tabular}{lp{9.7cm}}
		O1: & Reflexivity: $x\leq x$,\\
		O2: & Transitivity: $(x\leq y\land y\leq z)\rightarrow x\leq z$,\\
		O3: & Antisymmetry: $(x\leq y \land y\leq x)\leftrightarrow x=y$,\\
		O4: & Totality: $(x\leq y\rightarrow\bot)\rightarrow y\leq x$.\\
		O5: & Addition order: $x\leq y\rightarrow x+z\leq y+z$.\\
		O6: & Multiplication order: $(x\leq y\land z\geq0)\rightarrow xz\leq yz$.\\		
	\end{tabular}
\end{axiom}

\begin{axiom}[Completeness Property of $\mathbb{\widehat{R}}$]
	\label{ax:CompletenessPropertyOfRHat}
	Every non-empty bounded above subset of $\mathbb{\widehat{R}}$ has a least upper bound (see Definition \ref{def:ClassicalLeastUpperBound}).
\end{axiom}

\begin{axiom}[Infinitesimal Property of $\mathbb{\widehat{R}}$]
	\label{ax:InfinitesimalPropertyOfRHat}
	The set $\mathbb{\widehat{R}}$ has an infinitesimal (see Definition \ref{def:NumbersInWidehatR} for what infinitesimal is).
\end{axiom}

\begin{axiom}[Archimedean Property of $\mathbb{\widehat{R}}$]
	\label{ax:ArchimedeanPropertyOfRHat}
	For all $x,y>0$, $\exists n $ such that $nx>y$ (see Definition \ref{def:Operator<>} for the operator $>$).
\end{axiom}


Notice what Axiom \ref{ax:CompletenessPropertyOfRHat} and Axioms \ref{ax:InfinitesimalPropertyOfRHat} and \ref{ax:ArchimedeanPropertyOfRHat} cause. The first axiom, which states the completeness property of $\mathbb{\widehat{R}}$, causes computability issues in our set. The last two axioms, they cause the consistency trouble. Infinity and infinitesimals are formally defined in \Cref{def:NumbersInWidehatR}.


\begin{definition}
	\label{def:NumbersInWidehatR}
	An element $x\in\mathbb{\widehat{R}}$ is
	\begin{itemize}
		\item infinitesimal iff  $\forall n\in\mathbb{N}$ $|x|<\frac{1}{n}$;
		\item finite iff  $\exists r\in\R$ $|x|<r$;
		\item infinite iff  $\forall r\in\R$ $|x|>r$;
		\item appreciable iff  $x$ is finite but not an infinitesimal;
	\end{itemize}
	By using notation $\epsilon$ as an infinitesimal, an infinity $\omega$ is defined as a reciprocal of $\epsilon$, i.e. $\frac{1}{\epsilon}$.
\end{definition}

\begin{definition}
	\label{def:Operator<>}
	For any numbers $x,y\in\mathbb{\widehat{R}}$,
	\begin{enumerate}
		\item $x\geq y:=x< y\rightarrow\bot$
		\item $x<y:=x\leq y\land(x=y\rightarrow\bot)$
		\item $x>y:=x\geq y\land(x=y\rightarrow\bot)$
	\end{enumerate}
\end{definition}

\noindent If we look further, the set $\mathbb{\widehat{R}}$ is actually an inconsistent set. \Cref{ex:InconsistentOfL} gives one of these contradictions.

\begin{example}
	\label{ex:InconsistentOfL}
	Suppose that we have a set $S=\{x\in\mathbb{\widehat{R}}:\abs{x}<\frac{1}{n}\textnormal{ for all } n\in\mathbb{N}\}$. In other words, the set $S$ consists of all infinitesimals in $\mathbb{\widehat{R}}$. It is easily proven that $S$ is not empty and bounded above. So by Completeness Axiom, $S$ has a least upper bound. Suppose that $z$ is its least upper bound (which also means that $z$ must be an infinitesimal). Because $z$ is an infinitesimal, $2z$ is also an infinitesimal and this means $2z$ is also in $\mathbb{\widehat{R}}$. By using \Cref{def:Operator<>}, $z<2z$ and so, $z$ is not the least upper bound of $S$. Now suppose that sup $S=2z$. The same argument can be used to show that $2z$ is not supremum of $S$ but $3z$. We can build this same argument infinitely to show that there does not exist $s$ such that sup $S=s$. Thus, we have $\exists s:\textnormal{sup } S=s$ and $\nexists s:\textnormal{sup } S=s$.
\end{example}

This kind of contradiction forces us to use a non-explosive logic such as paraconsistent logic instead of classical logic to do our reasoning in $\mathbb{\widehat{R}}$. Furthermore, we choose a particular paraconsistent reasoning strategy, \textit{Chunk and Permeate} (C\&P), to resolve our dilemma. The detail explanations of this strategy can be seen in \cite{brown2004chunk}.

\subsection{Chunks in $\mathfrak{\widehat{L}}$ and The Creation of The Set $\Rzl$} 

Using the C\&P strategy, we divided $\mathfrak{\widehat{L}}$ into some consistent chunks --- naturally, there might be several ways to do it (e.g. one can have an idea to divide the original set into two, three, or even more chunks). Nevertheless, we found out that one particular way to have just two different chunks, as provided in here, is the most interesting one as it leads to the creation of a new model. One chunk is a set which contains Axioms (\ref{ax:AdditivePropertyOfRHat}--\ref{ax:TotalOrderPropertyOfRHat},\ref{ax:InfinitesimalPropertyOfRHat}), while the other chunk is a set which contains Axioms (\ref{ax:AdditivePropertyOfRHat}--\ref{ax:CompletenessPropertyOfRHat},\ref{ax:ArchimedeanPropertyOfRHat}). The consistency of each chunk was proved by providing a model for each of them.

\subsubsection{Model for The First Chunk}
One of the possible --- and interesting --- chunks is a set consists of Axioms \ref{ax:AdditivePropertyOfRHat}-\ref{ax:TotalOrderPropertyOfRHat} and Axiom \ref{ax:InfinitesimalPropertyOfRHat}. Here we proved the consistency of this chunk and also some corollaries that we have. 

It is well-established that the set of hyperreals, $\mathbb{^*R}$, clearly satisfied those axioms \cite{benci2003alpha}. Nevertheless, the construction of hyperreals $\mathbb{^*R}$ depends on highly non-constructive arguments. In particular, it requires an axiom of set theory, the well-ordering principle, which assumes into existence something that cannot be constructed \cite{henle2012which}. Here we proposed to take a look at a simpler set. Remember that our set has to contain not just $\R$, but also infinitesimals and infinities (and the combinations of the two). We take $\Rz$, functions from integers to real numbers, as our base set. The member of $\Rz$ consists of \textit{standard} and \textit{non-standard} parts. The \textit{standard part} of a certain number simply shows its finite element (the real part), while the \textit{non-standard part} shows its infinite or infinitesimal part (see \Cref{def:MemberOfR^Z}).

\begin{definition}[Member of $\Rz$]
	\label{def:MemberOfR^Z}
	A typical member of $\Rz$ has the form $\mathbf{x}=\langle\epsilon_{-i}, \widehat{x},{\epsilon_j}\rangle$ where $\widehat{x}\in\R$ and $\epsilon_n$ denotes the sequence of the constant part of infininitesimals if $n>0$, and infinities if $n<0$. 
\end{definition}

\noindent Notice that the symbol $\widehat{x}$ in Definition \ref{def:MemberOfR^Z} signs the standard part of a number in $\Rz$. Thus, the member of $\Rz$ can be seen as a sequence of infinite numbers. Example \ref{ex:ConvertNumberToR^Z} gives an overview of how to write a number as a member of $\Rz$.

\begin{example}[Numbers in $\Rz$]
	\label{ex:ConvertNumberToR^Z}
	\hspace{0cm}
	\begin{enumerate}
		\item The number 1 is written as $\textbf{1}=\langle\dots0,0,\widehat{1},0,0,\dots\rangle$.
		\item The number $\epsilon$ is written as $\mathbf{\epsilon}=\langle\dots0,\widehat{0},1,0,\dots\rangle$.
		\item The number $\omega$ (one of the infinities) is written as $\mathbf{\omega}=\langle\dots0,1,\widehat{0},0,\dots\rangle$.
		\item The number $2+2\epsilon-\omega^2$ is written as $\mathbf{2+2\epsilon-\omega^2}=\langle\dots,0,-1,0,\widehat{2},2,0,\dots\rangle$.
	\end{enumerate}
\end{example}

By using this form, all of the possible numbers can be written in $\Rz$. However, this infinite form is problematic in a number of ways. For example, multiplication cannot be easily defined and there might exist multiple inverses if the set $\Rz$ was going to be used (see Example 3.16 in \cite[p.~47]{nugraha2018naive}). Because of this, the semi-infinite form is motivated and the modified set is denoted by $\Rzl$. The only difference between $\Rzl$ and $\R$ is that, for any number $\mathbf{x}$, we will not have an infinite sequence on the left side of its standard part. See \cref{ex:ConvertNumberToR^Z<} and compare to \Cref{ex:ConvertNumberToR^Z}.

\begin{example}[Numbers in $\Rzl$]
	\label{ex:ConvertNumberToR^Z<}
	\hspace{0cm}
	\begin{enumerate}
		\item A number 1 is written as $\textbf{1}=\langle\widehat{1},0,0,\dots\rangle$.
		\item A number $\epsilon$ is written as $\mathbf{\epsilon}=\langle\widehat{0},1,0,\dots\rangle$.
		\item A number $\omega$ (one of the infinities) is written as $\mathbf{\omega}=\langle1,\widehat{0},0,\dots\rangle$.
		\item A number $2+2\epsilon-\omega^2$ is written as $\mathbf{2+2\epsilon-\omega^2}=\langle-1,0,\widehat{2},2,0,\cdots\rangle$.
	\end{enumerate}
\end{example}

\noindent Some of the properties of the set $\Rzl$ are as follows, while their proof (when needed) and some examples of them can be seen in \cite[p.~51--57]{nugraha2018naive}.

\begin{definition}[Addition and Multiplication in $\Rzl$]
	\label{def:AdditionAndMultiplicationInR^<Z}
	For any number $\mathbf{x}=\langle x_z\rangle$ and $\mathbf{y}=\langle y_z\rangle$ in $\Rzl$, define:
	\begin{align*}
		\mathbf{x}+\mathbf{y} = \langle x_z+y_z:z\in\mathbb{Z}\rangle
	\end{align*}
	and $\mathbf{x}\times\mathbf{y}$ is calculated by:
	\begin{align*}
		\mathbf{x}\times\mathbf{y} = \left(\sum_{i=-m} a_i\epsilon^{i}\right)\widehat{\times}\left(\sum_{j=-n} b_j\epsilon^{j}\right)=	\left(\sum_{k\in\mathbb{Z}}c_k\epsilon^{k}\right),
	\end{align*}
	where $c_k=\sum_{i+j=k}a_ib_j$.
\end{definition}

\begin{definition}[Order in $\Rzl$]
	The set $\Rzl$ is endowed with $\widehat{\leq}$, the lexicographical ordering.
\end{definition}

\begin{proposition}
	\label{prop:AdditiveAxiomR^<Z}
	The set $\Rzl$ satisfies the additive property in Axiom \ref{ax:AdditivePropertyOfRHat}.
\end{proposition}


\begin{proposition}
	\label{prop:MultiplicativeAxiomR^<Z}
	The set $\Rzl$ satisfies the multiplicative property in Axiom \ref{ax:MultiplicativePropertyOfRHat}.
\end{proposition}


\begin{proposition}
	\label{prop:DistributiveAxiomR^<Z}
	The set $\Rzl$ satisfies the distributive property in Axiom \ref{ax:DistributivePropertyOfRHat}.
\end{proposition}

\begin{proposition}
	\label{prop:OrderAxiomR^<Z}
	The set $\Rzl$ satisfies the total order property in Axiom \ref{ax:TotalOrderPropertyOfRHat}.
\end{proposition}

The following results show how to find an inverse of any members of $\Rzl$ and its uniqueness property.

\begin{proposition}
	\label{prop:InverseOf1+Omega}
	The number $\mathbf{1+\omega}$ has a unique inverse.
\end{proposition}

\begin{proposition}
	\label{prop:InverseOfEpsilon+Omega}
	The number $\mathbf{\omega\widehat{+}\epsilon}$ has a unique inverse.
\end{proposition}

\begin{lemma}
	\label{lem:InverseREpsilon+SOmega}
	For any $r,s\in\R$, a number $r\mathbf{\epsilon}\widehat{+}s\mathbf{\omega}$ has a unique inverse.
\end{lemma}

\begin{lemma}
	\label{lem:InverseRPlusOmega}
	For any $r\in\R$, a number $r\widehat{+}\mathbf{\omega}$ (or $r\widehat{+}\mathbf{\epsilon})$ has a unique inverse.
\end{lemma}

\begin{theorem}
	For any number $\mathbf{x}\in\Rzl$, $\mathbf{x}$ has a unique inverse.
\end{theorem}

\subsubsection{Model for The Second Chunk}

The most evident model for this second chunk is the set of real numbers, $\mathbb{R}$. Thus, so far, we have already had two chunks in $\mathfrak{\widehat{L}}$ and we proved their consistencies by providing a model for each of them.

\subsection{Grossone Theory and The Set $\Rzl$} 

Theories that contain infinities have always been an issue and have attracted much research, for example \cite{cantor1915contributions,godel2016consistency,hardy2015orders,mayberry2000foundations,robinson1974non}. Note that the arithmetic developed for infinite numbers was quite different with respect to the finite arithmetic that we are used to dealing with. For example, Sergeyev in \cite{sergeyev2013arithmetic} created the Grossone theory. The basic idea of this theory is to treat infinity as an `normal' number, so that our usual arithmetic rules apply. He named this infinite number \textit{Grossone} and denoted it with $\1$. The four axioms that form this Grossone theory and its details can be seen \cite{sergeyev2013arithmetic}.

It is very important to emphasis that \1 is a number, and so it works as a usual number. For example, there exist numbers such as $\1-100,\1^3+16,\ln\1$, and etc. Also for instance, $\1-1<\1$.\footnote{This is unlike the way the usual infinity, $\infty$, behaves where for example, $\infty-1=\infty$. It also differs from how Cantor’s cardinal numbers behave.} The introduction for this new number \1 makes us able to rewrite the set of natural numbers $\mathbb{N}$ as:
\begin{center}
	$\mathbb{N}=\{1,2,3,4,5,6,\dots,\1-2,\1-1,\1\}$.
\end{center}
\noindent Furthermore, adding the Infinite Unit Axiom (IUA) to the axioms of natural numbers will define the set of extended natural numbers $\mathbb{^*N}$:
\begin{center}
	$\mathbb{^*N}=\{1,2,3,4,5,6,\dots,\1-1,\1\,\1+1,\dots,\1^2-1,\1^2,\dots\}$,
\end{center}
and the set $\mathbb{^*Z}$, extended integer numbers, can be defined from there. 

Here we argued that our new set $\Rzl$ provides the model of Grossone theory and therefore proves its consistency rather in a deftly way.

\begin{definition}
	In our system $\Rzl$, the number \1 is written as:
	\begin{center}
		$\1=\langle1,\widehat{0},0,\dots\rangle$.
	\end{center}
\end{definition}

\begin{proposition}
	For every finite number $\mathbf{r}\in\Rzl$, $\mathbf{r}<\1$.
\end{proposition}
\begin{proof}
	The order in set $\Rzl$ is defined lexicographically. Now suppose that $\mathbf{r}=\langle0,\widehat{r},0,0,\dots\rangle$ where $\widehat{r}\in\R$. Then it is clear that $\mathbf{r}<\1$.
\end{proof}

\begin{proposition}
	All of the equations in the Identity Axiom of Grossone theory are also hold in $\Rzl$.
\end{proposition}

\noindent The fractional form of $\1$ 
can also be defined in $\Rzl$ as:
\begin{center}
	for any $n\in\mathbb{N}$, $\frac{\1}{n}=\langle\frac{1}{n},\widehat{0},0,\dots\rangle$.
\end{center}

\noindent Speaking about the inverse, one of the advantages of having the set $\Rzl$ is to be able to see what the inverse of a number looks like, not like in the Grossone theory. See Example \ref{ex:InverseComparison} for more details.

\begin{example}
	\label{ex:InverseComparison}
	In Grossone theory, the inverse of $\frac{1}{\1}+\1$ is just $\frac{1}{\frac{1}{\1}+\1}$. While in our set $\Rzl$, $\frac{1}{\1}+\1$ is written as $\epsilon+\omega$ and its inverse is
	\begin{center}
		$\langle\widehat{0},1,0,-1,0,1,0,-1,0,\dots\rangle=\epsilon-\epsilon^3+\epsilon^5-\epsilon^7+\dots$.
	\end{center}
	In other words, the more explicit form of $\frac{1}{\frac{1}{\1}+\1}$ is a series $(-1)^{n+1}\1^{-(2n-1)}$ for $n=1,2,3,\dots\in\mathbb{N}$.
\end{example}

In \cite{lolli2015metamathematical}, Gabriele Lolli analysed and built a formal foundation of the Grossone theory based on Peano's second order arithmetic. He also gave a slightly different notion of some axioms that Sergeyev used. One of the important theorems in Lolli's paper is the proof that Grossone theory -- or at least his version of it -- is consistent. However, as he said also in \cite{lolli2015metamathematical}, `The statement of the theorem is of course conditional, as apparent from the proof, upon the consistency of $\text{PA}_\mu^2$' while `its model theoretic proof is technically rather demanding'. 

Thus, through what presented in this subsection, we have proposed a new way to prove the consistency of Grossone theory by providing a \textit{straightforward} model of it. There is no need for complicated model-theoretic proofs. The set $\Rzl$ is enough to establish the consistency of Grossone theory in general. Moreover, the development in the next sessions can also be seen, at least in part, as a contribution to the development of Grossone theory.

\section{Topology on The Set $\Rzl$}
\label{sec:TopologyOnRZ<}

Some topological properties of the set $\Rzl$ are discussed in this section. However, there are a number of definitions and issues that should be addressed first in order to understand how those properties will be applied to our set properly.

\subsection{Metrics in $\Rzl$}

We defined what is meant by a distance (metric) between each pair of elements of $\Rzl$.

\begin{definition}
	\label{def:metric}
	A \textit{metric} $\rho$ in a set $X$ is a function
	\begin{align*}
		\rho:X\times X\rightarrow\left[0,\infty\right)
	\end{align*}
	where for all $x,y,z\in X$, these four conditions are satisfied:
	\begin{enumerate}
		\item $\rho(x,y)\geq0$,
		\item  $\rho(x,y)=0$ if and only if $x=y$,
		\item $ \rho(x,y)=d(y,x)$,
		\item $\rho(x,z)\leq \rho(x,y)+\rho(y,z)$.
	\end{enumerate}
	When that function $\rho$ satisfies all of the four conditions above except the second one, $\rho$ is called a \textit{pseudo-metric}\footnote{It is not without reason that we introduced the concept of the pseudo-metric here. This kind of metric will make sense when we are in $\R$. For example, the distance between $0$ and $\epsilon$ is $0$ as our lens is not strong enough to distinguish those two numbers in $\R$.} on $X$. 
\end{definition}

Now we define two functions $\dis$ and $\diss$ in $\Rzl$ as follows:

\begin{definition}
	\label{def:metricAndPseudoMetric}
	For all $\mathbf{x},\mathbf{y}\in\Rzl$,
	\begin{align*}
		\dis:\Rzl\times\Rzl\rightarrow\Rzl \textnormal{ and } \diss:\Rzl\times\Rzl\rightarrow\mathbb{R}
	\end{align*}
	where $\dis(\mathbf{x},\mathbf{y})=\abs{\mathbf{y}-\mathbf{x}}$ and $\diss(\mathbf{x},\mathbf{y})=\texttt{St}(\abs{\mathbf{y}-\mathbf{x}})$.
\end{definition}

It can be easily verified that $\dis$ is a metric in $\Rzl$ (and so $\left(\Rzl,\dis\right)$ forms a metric space) and $\diss$ is a pseudo-metric in $\Rzl$ (and so $\left(\Rzl,\diss\right)$ forms a pseudo-metric space). 

\subsection{Balls and Open Sets in $\Rzl$}
Now that we have the notion of distance in $\Rzl$, we can define what it means to be an open set in $\Rzl$ by first defining what a \textit{ball} is in $\Rzl$.

\begin{definition}
	A \textit{ball} of radius $\mathbf{y}$ around the point $\mathbf{x}\in\Rzl$ is
	\begin{align*}
		B_{\mathbf{x}}(\mathbf{y})=\{\mathbf{z}\in\Rzl\mid d_1(\mathbf{x},\mathbf{z})<\mathbf{y}\},
	\end{align*}
	where $d_1(\mathbf{x},\mathbf{y})$ is either $\dis(\mathbf{x},\mathbf{y})$ or $\diss(\mathbf{x},\mathbf{y})$.
\end{definition}

\noindent We require this additional definition in order to set forth our explanation about balls properly:

\begin{definition}
	\label{def:delta}
	The sets $\Delta^m$ and $\Delta^{\underdownarrow{m}}$ are defined as follows:
	\begin{center}
		$\Delta^m=\{\mathbf{x}:\mathbf{x}=a_m\mathbf{\epsilon}^m\}$ and $\Delta^{\underdownarrow{m}}=\bigcup\limits_{n\geq m}\Delta^n$,
	\end{center}
	where $m\in\mathbb{N}\cup\{0\}$, $a_m\in\R$ and $a_m\neq0$ whenever $m\geq1$. 
\end{definition}

We have to be careful here as unlike in classical topology, there are different notions of balls that can be described as follows. The \textit{first} possible notion of balls is when we use $\dis$ as our metric and having $\mathbf{y}>0$ as our radius. In this case, in $\Rzl$, the ball around a point $\mathbf{x}$ with $\mathbf{y}$ radius is an interval $(\mathbf{x}-\mathbf{y},\mathbf{x}+\mathbf{y})$. Note that by using $\mathbf{y}$ as the radius, beside having the usual balls with ``real'' radius (\textit{St-balls}) (that is when $\mathbf{y}=r\in\R$), we also have some infinitesimally small balls (\textit{e-balls}) when $\mathbf{y}=\einf\in\Delta^{\underdownarrow{m}}$ for any given $m$. The \textit{second} possible notion is while we use the same metric $\dis$, we have $\nicefrac{1}{n}$ for some $n\in\mathbb{N}$ as its radius. This produces balls (\textit{rat-balls}) in the form of $\left(\mathbf{x}-\nicefrac{1}{n},\mathbf{x}+\nicefrac{1}{n}\right)$. The \textit{third} possibility is by using $\diss$ as our metric. In this case, interestingly, the balls around a point $\mathbf{x}$ with $\nicefrac{1}{n}$ radius will be in the form of the following set: 
\begin{center}
	$\{\mathbf{y}\mid\texttt{St}(\mathbf{y})\in(\texttt{St}(\mathbf{x})-\nicefrac{1}{n},\texttt{St}(\mathbf{x})+\nicefrac{1}{n})\}$.
\end{center}

\noindent We call this kind of balls as \textit{psi-balls}. See Table \ref{table:ballsinRandRz<} for the summary of these possibilities of balls in our sets.

\begin{table}
	\caption{Some types of balls both in $\R$ and $\Rzl$}
	\begin{tabular}{ |l|l|l| }
		\hline
		\multicolumn{3}{|c|}{Balls in $\R$ and $\Rzl$} \\
		\hline
		The Set & The Metric & The Form of The Balls\\
		\hline
		& &\\
		$\R$   & $\rho(x,y)$\hspace{0.29cm}$=\abs{x-y}$ & $B_x(r)$\hspace{0.33cm}$=(x-r,x+r)$\\
		$\Rzl$&   $\dis(\mathbf{x},\mathbf{y})$\hspace{0.26cm}$=\abs{\mathbf{x}-\mathbf{y}}$ & $B_\mathbf{x}(\mathbf{r})$\hspace{0.32cm}$=(\mathbf{x}-\mathbf{r},\mathbf{x}+\mathbf{r})$\\
		$\Rzl$ & $\dis(\mathbf{x},\mathbf{y})$\hspace{0.26cm}$=\abs{\mathbf{x}-\mathbf{y}}$ & $B_\mathbf{x}(\nicefrac{1}{n})=(\mathbf{x}-\nicefrac{1}{n},\mathbf{x}+\nicefrac{1}{n})$\\
		$\Rzl$  & $\diss(\mathbf{x},\mathbf{y})=\texttt{St}(\abs{\mathbf{x}-\mathbf{y}})$ &
		$B_\mathbf{x}(\nicefrac{1}{n})=B_{\texttt{St}(\mathbf{x})}(\nicefrac{1}{n})=\{\mathbf{y}\mid\texttt{St}(\mathbf{y})\in$\\
		& & \hspace{1.15cm}$\quad(\texttt{St}(\mathbf{x})-\nicefrac{1}{n},\texttt{St}(\mathbf{x})+\nicefrac{1}{n})\}$\\
		& &\\
		\hline
	\end{tabular}
	\label{table:ballsinRandRz<}
\end{table}
	
\begin{remark}
	In $\R$, the $\einf$-ball does not exist, whereas in $\Rzl$ there are infinitely many $\einf$-balls around every point there.
\end{remark}

Finally, we defined what it means to be an open set in $\Rzl$. Notice that because we had two notions of ball in our set, i.e. \texttt{St}-balls and $\einf$-ball, it led us to two different notions of openness as follows.

\begin{definition}[\texttt{St}-open]
	\label{def:St-open}
	A subset $O\subseteq\Rzl$ is \textit{St-open} iff
	\begin{center}
		$\forall x\in O$ $\exists n\in\mathbb{N}\textnormal{ s.t. }B_x\left(\frac{1}{n}\right)\subseteq O$.
	\end{center}
\end{definition}

\begin{definition}[$\einf$-open]
	\label{def:e-open}
	A subset $O\subseteq\Rzl$ is \textit{e-open} iff
	\begin{center}
		$\forall x\in O$ $\exists\einf\in\Delta^{\underdownarrow{m}}\textnormal{ s.t. }B_x(\einf)\subseteq O$.
	\end{center}
\end{definition}

\noindent Remember that the set $\Delta^{\underdownarrow{m}}$ is defined in Definition \ref{def:delta}.

\begin{example}
	The interval $(\textbf{2},\textbf{3})$ in $\Rzl$ is \texttt{St}-open and also $\einf$-open.
\end{example}

\begin{example}
	\label{ex:nopenisnotequivwitheinfopen}
	The interval $(\textbf{0},\mathbf{\einf})$ in $\Rzl$ is $\einf$-open, but not \texttt{St}-open.
\end{example}

Example \ref{ex:nopenisnotequivwitheinfopen} gives us the theorem below:


\begin{theorem}
	For any set $U\subseteq\Rzl$,
	\begin{align*}
	\Rzl\not\models \textit{If } U \textit{ is \texttt{St}-open, then } U \textit{ is } \einf \textit{-open}.
	\end{align*}
\end{theorem}

Using the two definition of openness given in Definitions \ref{def:St-open} and \ref{def:e-open}, we defined what it means by two points are topologically distinguishable. There are also two different notions of distinguishable points as can be seen in Definitions \ref{def:St-distinguishable} and \ref{def:e-distinguishable}.

\begin{definition}[\texttt{St}-distinguishable]
	\label{def:St-distinguishable}
	Any two points in $\Rzl$ are \textit{St-distinguishable} if and only if there is a \texttt{St}-open set containing precisely one of the two points.
\end{definition}

\begin{definition}[$\einf$-distinguishable]
	\label{def:e-distinguishable}
	Any two points in $\Rzl$ are \textit{e-distinguishable} if and only if there is an $\einf$-open set containing precisely one of the two points.
\end{definition}

\subsection{Topological Spaces}

\begin{definition}
	\label{def:topology}
	Let $X$ be a non-empty set and $\uptau$ a collection of subsets of $X$ such that:
	\begin{enumerate}
		\item $X\in\uptau$,
		\item $\emptyset\in\uptau$,
		\item If $O_1,O_2,\dots,O_n\in\uptau$, then  $\bigcap_{k=1}^nO_k\in\uptau$,
		\item If $O_\alpha\in\uptau$ for all $\alpha\in A$, then $\bigcup_{\alpha\in A} O_\alpha\in\uptau$.
	\end{enumerate}
	The pair of objects $(X,\uptau)$ is called a \textit{topological space} where $X$ is called the \textit{underlying set}, the collection $\uptau$ is called the \textit{topology} in $X$, and the members of $\uptau$ are called \textit{open sets}.
\end{definition}

Note that if $\uptau$ is the collection of open sets of a metric space $(\mathcal{X},\rho)$, then $(\mathcal{X},\uptau)$ is a \textit{topological metric space}, i.e. a topological space associated with the metric space $(X,\rho)$. There are at least three interesting topologies in $\Rzl$ as can be seen in Definition \ref{def:topologiesonRz<} below.

\begin{definition}
	\label{def:topologiesonRz<}
	The standard topology $\uptau_\textnormal{St}$ on the set $\Rzl$ is the topology generated by all unions of \texttt{St}-balls. The $\einf$-topology in $\Rzl$, $\uptau_\einf$, is the topology generated by all unions of $\einf$-balls and the third topology in $\Rzl$ is pseudo-topology, $\uptau_\psi$, when it is induced by $\diss$.
\end{definition}

\begin{axiom}
	\label{axiom:RzformstopologicalMetricSpace}
	$(\Rzl,\uptau_n),$ $(\Rzl,\uptau_\einf),$ and $(\Rzl,\uptau_\psi)$ form topological metric space with $\dis$ as their metrics (for the first two) and $\diss$ for the third one.
\end{axiom}

\begin{theorem}
	$(\Rzl,\uptau_n)$ is not a Hausdorff space but it is a preregular space.\footnote{Hausdorff space and preregular space are defined as usual.}
\end{theorem}
\begin{proof}
	$\Rzl$ does not form a Hausdorff space because under the topology $\uptau_\textnormal{St}$, there are two distinct points, $\mathbf{\epsilon}=\langle \widehat{0},1\rangle$ and $\mathbf{\epsilon}+\textbf{1}=\langle\widehat{1},1\rangle$ for example, which are not neighbourhood-separable. It is impossible to separate those two points with \texttt{St}-ballss as $\nicefrac{1}{n}>\einf$ for every $n\in\mathbb{N}$ and $\einf\in\Delta^{\underdownarrow{m}}$. However, it is a preregular space as every pair of two \texttt{St}-distinguishable points in $\Rzl$ can be separated by two disjoint neighbourhoods. This follows directly from Definition \ref{def:St-distinguishable}.
\end{proof}

\begin{theorem}
	$(\Rzl,\uptau_\einf)$ is a non-connected space and it forms a Hausdorff space.
\end{theorem}
\begin{proof}
	We observe that for all $\mathbf{x_0}\in\Rzl$ and $\einf\in\Delta^{\underdownarrow{m}}$, the balls $B_\einf(\mathbf{x_0})$ are $\einf$-open and so is the whole space. To show that $\Rzl$ is not connected, let
	\begin{align*}
		S_1 & =\{\mathbf{x}\in\Rzl\mid (\mathbf{x}\leq\textbf{0})\textnormal{ or } (\mathbf{x}>\textbf{0} \textnormal{ and } \mathbf{x}\in\Delta^{\underdownarrow{m}})\} \textnormal{ and }\\
		S_2 & =\{\mathbf{x}\in\Rzl\mid (\mathbf{x}>\textbf{0})\textnormal{ and } \mathbf{x}\notin\Delta^{\underdownarrow{m}}\}.
	\end{align*}
	The sets $S_1$ and $S_2$ are $\einf$-open, disjoint and moreover, we have that $\Rzl=S_1\cup S_2$ (and so $\Rzl$ is not connected). For any $\mathbf{x},\mathbf{y}\in\Rzl$, $B_x(\nicefrac{\dis(x,y)}{2})$ and $B_y(\nicefrac{\dis(x,y)}{2})$ are open and disjoint. Thus, $\Rzl$ forms a Hausdorff space.
\end{proof}

We will now state the usual definition of the basis of a topology $\uptau$.

\begin{definition}
	Let $(X,\uptau)$ be a topological space. A \textit{basis} for the topology $\uptau$ is a collection $\mathcal{B}$ of subsets from $\uptau$ such that every $U\in\uptau$ is the union of some collections of sets in $\mathcal{B}$, i.e.
	\begin{center}
		$\forall U\in\uptau$, $\exists\mathcal{B^*}\subseteq\mathcal{B}$ s.t. $U=\bigcup\limits_{B\in\mathcal{B^*}}B$
	\end{center}
\end{definition}
%

\begin{example}
	On $\R$ with its usual topology, the set $\mathcal{B}=\{(a,b):a<b\}$ is a topological basis.
\end{example}

\begin{definition}
	Let $(X,\uptau)$ be a topological space and let $x\in X$. A local basis of $x$ is a collection of open neighbourhoods of $x$, $\mathcal{B}_x$, such that for all $U\in\uptau$ with $x\in U$, $\exists B\in\mathcal{B}_x$ such that $x\in B\subset U$.
\end{definition}

\begin{definition}
	Let $(X,\uptau)$ be a topological space. Then $(X,\uptau)$ is first-countable if every point $x\in X$ has a countable local basis.
\end{definition}

\begin{definition}
	Let $(X,\uptau)$ be a topological space. Then $(X,\uptau)$ is second-countable if there exists a basis $\mathcal{B}$ of $\uptau$ that is countable.
\end{definition}

\begin{theorem}
	$(\Rzl,\tau_\einf)$ is first countable but not second-countable.\footnote{Note that the space $(\Rzl,\tau_n)$ is still second-countable.}
\end{theorem}
\begin{proof}
From Axiom \ref{axiom:RzformstopologicalMetricSpace} and because every metric space is first-countable, it follows that $(\Rzl,\tau_\einf)$ is first-countable. However, there cannot be any countable bases in $\tau_\einf$ as the uncountably many open sets $O_x=(x-\einf,x+\einf)$ are disjoint.
\end{proof}

\section{Calculus on $\Rzl$}
\label{sec:CalculusOnRZ<}


It has been proved previously that the set $\Rzl$ forms a field. Remember that for any $\mathbf{x}\in\Rzl$,
\begin{center}
	$\mathbf{x}=\langle x_{-n},x_{-(n-1)},\dots,x_{-2},x_{-1},\widehat{x},x_1,x_2,x_3,\dots\rangle$
\end{center} 
where
\begin{center}
	\texttt{St}$(\mathbf{x})=\widehat{x}$,\\
	\texttt{Nst}$_\epsilon(\mathbf{x})=\{x_1,x_2,x_3,\dots\}$, and\\
	\texttt{Nst}$_\omega(\mathbf{x})=\{x_{-n},x_{-(n-1)},\dots x_{-1}\}$.
\end{center}
\noindent In other words, for every $\mathbf{x}\in\Rzl$,
\begin{center}
	$\mathbf{x}=$ \texttt{Nst$_\omega$}$(\mathbf{x})$ $+$ \texttt{St}$(\mathbf{x})$ $+$ \texttt{Nst$_\epsilon$}$(\mathbf{x})$.
\end{center}

Note that we can think of \texttt{St}(), \texttt{Nst$_\epsilon$}(), and \texttt{Nst$_\omega$}() as linear functions -- that is for any $\mathbf{x},\mathbf{y}\in\Rzl$ and a constant $c\in\R$,
\begin{center}
	\texttt{St}$(\mathbf{x}+\mathbf{y})=$\texttt{St}$(\mathbf{x})+$\texttt{St}$(\mathbf{y})$,
	\texttt{St}$(c\mathbf{x})=c$\texttt{St}$(\mathbf{x})$,\\
	\texttt{Nst$_\epsilon$}$(\mathbf{x}+\mathbf{y})=$\texttt{Nst$_\epsilon$}$(\mathbf{x})+$\texttt{Nst$_\epsilon$}$(\mathbf{y})$, \texttt{Nst$_\epsilon$}$(c\mathbf{x})=c$\texttt{NSt$_\epsilon$}$(\mathbf{x})$,\\
	\texttt{Nst$_\omega$}$(\mathbf{x}+\mathbf{y})=$\texttt{Nst$_\omega$}$(\mathbf{x})+$\texttt{Nst$_\omega$}$(\mathbf{y})$, and \texttt{Nst$_\omega$}$(c\mathbf{x})=c$\texttt{Nst$_\omega$}$(\mathbf{x})$.\\
	
\end{center}

\begin{definition}
	\label{def:microstablefunctioninR^Z<}
	Suppose that $\texttt{ni}_\epsilon(\mathbf{x})$ denotes the non-infinitesimal part of $\mathbf{x}\in\Rzl$, i.e. $\texttt{ni}_\epsilon(\mathbf{x})=\texttt{Nst}_\omega(\mathbf{x})+\texttt{St}(\mathbf{x})$ and function in $\Rzl$ be defined in the usual way. Then a function $f$ in $\Rzl$ is \textit{microstable} if and only if
	\begin{center}
		$\texttt{ni}_\epsilon(f(x+\epsilon))=\texttt{ni}_\epsilon(f(x))$,
	\end{center}
\end{definition}

\begin{example}
	Suppose that a function $f$ in $\Rzl$ is defined as follows:
	\[ f(\mathbf{x})=
	\begin{cases} 
	1, & \text{if } \texttt{St}(\mathbf{x})>0 \\
	0, & \text{else.}
	\end{cases}
	\]
	Then $f(\mathbf{x})$ is a microstable function.
\end{example}

\begin{theorem}
	Microstability is closed under addition, multiplication, and composition.\footnote{The proof of this theorem can be seen in \cite[p.~73]{nugraha2018naive}.}
\end{theorem}

Now for every function $f$ defined in $\Rzl$, we are going to have the operator $\texttt{Der}_f$ which takes a 2-tuple in $\left(\mathbb{R}\times\Rzl\right)$ as its input and returns a member of $\Rzl$ as the output, i.e.:
\begin{center}
	$\texttt{Der}_f:\R\times\Rzl\rightarrow \Rzl$.
\end{center}
Eventually, this operator will be called a \textit{derivative} of $f$.

Using Newton's original definition (and a slight change of notation), if a function $f(x)$ is differentiable, then its derivative is given by:
\begin{equation}
\label{Eq:Newton'sDefinition}
\texttt{Der}_f(\mathbf{x},\mathbf{\epsilon})=\cfrac{f(\mathbf{x}+\mathbf{\epsilon})-f(\mathbf{x})}{\mathbf{\epsilon}}.
\end{equation}

Now suppose that we want to find a derivative of $f$ where $f$ is a function defined in $\Rzl$. We can certainly use \Cref{Eq:Newton'sDefinition} to calculate it as that equation holds for any function $f$. But how is this calculation related to the calculus practised in classical mathematics? Note that using Newton's definition to calculate the derivative will necessarily involve an inconsistent step. This inconsistency is located in the treatment given to the infinitesimal number. Thus it makes sense that in order to explore the problem posed above, we will use a paraconsistent reasoning strategy which is called Chunk and Permeate.

\subsection{Chunk and Permeate for Derivative in $\Rzl$ }

Details on the Chunk \& Permeate reasoning strategy can be seen in \cite{brown2004chunk}. Before applying this strategy for the derivative in $\Rzl$, define a set $E$ which consists of any algebraic terms such that they satisfy:
\begin{center}
	$\texttt{St}(\texttt{Der}_f(\mathbf{x},\mathbf{\epsilon}))=f\textnormal{ }'(x)$,
\end{center}
where $f\textnormal{ }'(x)$ denotes the usual derivative of $f$ in $\R$. We will need this set $E$ when we try to define the permeability relation between chunks.

\begin{proposition}
	The set $E$ as defined above is inhabited.
\end{proposition}
\begin{proof}
	We want to show that the set $E$ has at least one element in it. It is clear that the identity function $\texttt{id}(x)=x$ is in $E$ because for all $\epsilon$:
	\begin{align*}
	\texttt{St}(\texttt{Der}_x(\mathbf{x},\mathbf{\epsilon})) &= \texttt{St}\left(\cfrac{\mathbf{x}+\mathbf{\epsilon}-\mathbf{x}}{\mathbf{\epsilon}}\right)\\
	&= \texttt{St}\left(\cfrac{\mathbf{\epsilon}}{\mathbf{\epsilon}}\right)\\
	&= \texttt{St}(\textbf{1})=1=f\textnormal{ }'(x).
	\end{align*}
\end{proof}

\begin{theorem}
	\label{thm:differentiablity}
	If $f$ and $g$ are microstable functions in $E$ and $c$ is any real constant, then
	\begin{enumerate}
		\item $f\pm g$ are in $E$,
		\item $cf$ is in $E$,
		\item $fg$ is in $E$,
		\item $\cfrac{f}{g}$ is in $E$, and
		\item $f\circ g$ is in $E$.
	\end{enumerate}
\end{theorem}

\noindent The proof of the above theorem is rather long and so can be seen in \cite[p.~76]{nugraha2018naive}.

Now we are ready to construct the chunk and permeate structure, called $\mathfrak{\widehat{R}}$, which is formally written as $\mathfrak{\widehat{R}}=\langle\{\Sigma_S,\Sigma_T\},\rho,T\rangle$ where the source chunk $\Sigma_S$ is the language of $\Rzl$, the target chunk $\Sigma_T$ is the language of $\R$, and $\rho$ is the permeability relation between $S$ and $T$.

\paragraph{The source chunk $\Sigma_S$}

As stated before, this chunk is actually the language of the set $\mathbb{{R^{Z_<}}}$ and therefore, it consists of all six of its axioms. The source chunk requires one additional axiom to define what it means by derivative. This additional axiom can be stated as:

\begin{center}
	S1: $Df=\texttt{Der}_f(\mathbf{x},\mathbf{\epsilon})$
\end{center}

\noindent where $\texttt{Der}_f(\mathbf{x},\mathbf{\epsilon})$ is defined in \Cref{Eq:Newton'sDefinition}.

\paragraph{The target chunk $\Sigma_T$} Again, the target chunk contains the usual axiom for the set or real numbers, $\R$. There is only one additional axiom needed for this chunk:
\begin{center}
	T1: $\forall \mathbf{x}$ $\mathbf{x}=\texttt{St}(\mathbf{x})$.
\end{center}
Note that the axiom T1 above is actually equivalent to saying that $\forall \mathbf{x}$ $\texttt{Nst}(\mathbf{x})=0$.

\paragraph{The permeability relation} The permeability relation $\rho(S,T)$ is the set of equations of the form 
\begin{center}
	$Df=g$
\end{center}
where $f\in E$. The function $g$ which is permeated by this permeability relation will be the first derivative of $f$ in $\mathbb{R}$. This permeability relation shows that the derivative notion is permeable to the set $\mathbb{R}$.

\begin{example}
	Suppose that $f(\mathbf{x})=\textbf{3}\mathbf{x}$ for all $\mathbf{x}$. First, working within $\Sigma_S$, the operator $D$ is applied to $f$ such that:
	\begin{align*}
	Df & = \texttt{Der}_f(\mathbf{x},\mathbf{\epsilon})\\
	& = \cfrac{\textbf{3}(\mathbf{x}+\mathbf{\epsilon})-\textbf{3}\mathbf{x}}{\mathbf{\epsilon}}\\
	& = \cfrac{\textbf{3}\mathbf{\epsilon}}{\mathbf{\epsilon}}=\textbf{3}.
	\end{align*}
	Note that $\texttt{St}(\texttt{Der}_f(\mathbf{x},\mathbf{\epsilon}))=\texttt{St}(\textbf{3})=3=f\textnormal{ }'(x)$, and so $f(x)\in E$. Permeating the last equation of $Df$ above to $\Sigma_T$ gives us:
	\begin{align*}
	Df = 3
	\end{align*}
	and so the derivative of $f(x)=3x$ is $3$.
\end{example}

\begin{example}
	Suppose that $f(\mathbf{x})=\mathbf{x}^\textbf{2}+\textbf{2}\mathbf{x}+\textbf{3}$ for all $\mathbf{x}$. First, working within $\Sigma_S$, the operator $D$ is applied to $f$ such that:
	\begin{align*}
	Df & = \texttt{Der}_f(\mathbf{x},\mathbf{\epsilon})\\
	& = \cfrac{(\mathbf{x}+\mathbf{\epsilon})^\textbf{2}+\textbf{2}(\mathbf{x}+\mathbf{\epsilon})+\textbf{3}-\mathbf{x}^\textbf{2}-\textbf{2}\mathbf{x}-\textbf{3}}{\mathbf{\epsilon}}\\
	& = \cfrac{\textbf{2}\mathbf{x}\mathbf{\epsilon}+\mathbf{\epsilon}^\textbf{2}+\textbf{2}\mathbf{\epsilon}}{\mathbf{\epsilon}}\\
	& = \textbf{2}\mathbf{x}+\mathbf{\epsilon}+\textbf{2}.
	\end{align*}
	Note that the standard part of $\textbf{2}\mathbf{x}+\mathbf{\epsilon}+\textbf{2}$ will depend on the domain of $\mathbf{x}$. That is:
	\[ \texttt{St}(\textbf{2}\mathbf{x}+\mathbf{\epsilon}+\textbf{2})=
	\begin{cases} 
	2x+2, & \text{if } x\in\R \\
	2, & \text{else.}
	\end{cases}
	\]
	In other words, if (and only if) $\texttt{Nst}(\mathbf{x})=0$, i.e. $x\in\R$, $Df$ can be permeated into $\Sigma_T$. Thus, if $x$ is a real number, then we have the derivative of $f(x)=x^2+2x+3=2x+2$.
\end{example}

\begin{example}
	Suppose that $f(\mathbf{x})=\sign(\mathbf{x})$ is defined as:
	\[ \sign(\mathbf{x})=
	\begin{cases} 
	\mathbf{1}, & \text{if } \texttt{St}(\mathbf{x})>0 \\
	\mathbf{0}, & \text{if } \texttt{St}(\mathbf{x})=0\\
	\mathbf{-1},& \text{if } \texttt{St}(\mathbf{x})<0
	\end{cases}
	\]
	First, working within $\Sigma_S$, the operator $D$ is applied to $f$ so that:
	\begin{align*}
	Df & = \texttt{Der}_f(\mathbf{x},\mathbf{\epsilon})\\
	& = \cfrac{\sign(\mathbf{x}+\mathbf{\epsilon})-\sign(\mathbf{x})}{\mathbf{\epsilon}}\\
	& = 0 \text{ (because $\forall x$ $\texttt{St}(\mathbf{x})=\texttt{St}(\mathbf{x}+\epsilon)$)}
	\end{align*}
	Note that $\texttt{St}(\texttt{Der}_f(\mathbf{x},\mathbf{\epsilon}))=\texttt{St}(\mathbf{0})=0=f\textnormal{ }'(x)$, and so $f(x)\in E$. Permeating the last equation of $Df$ above to $\Sigma_T$ gives us:
	\begin{align*}
	Df = 0
	\end{align*}
	and so the derivative of $f(\mathbf{x})=\sign(\mathbf{x})$ is $0$ for all $\mathbf{x}$. Notice that this is not the case in $\R$, where the derivative of the $\sign$ function at $x=0$ is not defined because of its discontinuity. However, this is not really a bizarre behaviour because if we look very closely at the infinitesimal neighbourhood of $\mathbf{x}$ when $\texttt{St}(\mathbf{x})=0$, the function $\sign(\mathbf{x})$ will look like a straight horizontal line and so it makes a perfect sense to have $0$ as the slope of the tangent line there. Moreover, this phenomenon also happens in distribution theory where $\sign$ function has its derivative everywhere.
\end{example}

\subsection{Transcendental Functions in $\Rzl$}

As we know, there are some special functions defined in real numbers and two of them are the trigonometric and the exponential functions. How then are these functions defined in $\Rzl$? Here we propose to define them using power series.

The first two trigonometric functions that we are going to discuss are the $\sin$ and $\cos$ functions. Using the MacLaurin power series, these two functions are defined as follows:
\begin{align}
\label{def:SinAsSeries}
\sin(\mathbf{x})=\sum\limits_{n=0}^{}\cfrac{(-1)^n}{(2n+1)!}\text{ }\mathbf{x}^{2n+1}
\end{align}
and
\begin{align}
\cos(\mathbf{x})=\sum\limits_{n=0}^{}\cfrac{(-1)^n}{(2n)!}\text{ }\mathbf{x}^{2n}.
\end{align}

\noindent The exponential function is defined as:
\begin{align}
\label{def:ExpAsSeries}
\exp(\mathbf{x})=\sum\limits_{n=0}^{}\cfrac{1}{n!}\mathbf{x}^n.
\end{align}

\noindent Note that the MacLaurin polynomial is just a special case of Taylor polynomial with regards to how the function is approximated at $\mathbf{x}=\textbf{0}$.

\begin{example}
	\label{ex:sinx+epsilon}
	Suppose that we have $\mathbf{x}=x+a\mathbf{\epsilon}=\langle \widehat{x},a,0,0,\dots\rangle$ where $x,a\in\R$. We want to know what $\sin(\mathbf{x})$ is. Based on \Cref{def:SinAsSeries},
	\begin{align*}
	\sin(\mathbf{x})=\sin(x+\mathbf{\epsilon})=(x+\mathbf{\epsilon})-\tfrac{\textbf{1}}{\textbf{3}!}(x+\mathbf{\epsilon})^\textbf{3}+\tfrac{\textbf{1}}{\textbf{5}!}(x+\mathbf{\epsilon})^\textbf{5}-\tfrac{\textbf{1}}{\textbf{7}!}(x+\mathbf{\epsilon})^\textbf{7}+\dots
	\end{align*}
	\noindent Our task now is to find all the members of $\texttt{Nst}_\mathbf{\epsilon}(\mathbf{\sin(x)})$ and also $\texttt{St}(\mathbf{\sin(x)})$. These are shown in \Cref{table:sinx+epsilon}. Note that from the way the $\sin$ function is defined, $x_i=0$ $\forall x_i\in\texttt{Nst}_\omega(\sin(\mathbf{x}))$. Thus from \Cref{table:sinx+epsilon}, we get:
	\begin{align*}
	\sin(\mathbf{x}) & = \sin(x+a\mathbf{\epsilon})\\
	& = \langle\widehat{\sin(x)},a\cos(x),-\tfrac{a^2}{2!}\sin(x),-\tfrac{a^\textbf{3}}{\textbf{3}!}\cos(x),\tfrac{a^4}{4!}\sin(x),\tfrac{a^\textbf{5}}{\textbf{5}!}\cos(x),\dots\rangle,
	\end{align*}
	and we also get
	\begin{align*}
	\sin(\mathbf{\epsilon})=\langle\widehat{0},\textbf{1},0,-\tfrac{\textbf{1}}{\textbf{3}!},0,\tfrac{\textbf{1}}{\textbf{5}!},\dots\rangle=\mathbf{\epsilon}-\tfrac{\textbf{1}}{\textbf{3}!}\mathbf{\epsilon}^\textbf{3}+\tfrac{\textbf{1}}{\textbf{5}!}\mathbf{\epsilon}^\textbf{5}-\dots
	\end{align*}
	for an infinitesimal angle $\mathbf{\epsilon}$.
	
\begin{table}
	\caption{$\texttt{St}(\sin(\mathbf{x)})$ and the first four members of $\texttt{Nst}_\epsilon(\sin(\mathbf{x}))$}
	\centering 
	\begin{tabular}{|l|l|l|}
		\hline 
		& Expanded Form	& Simplified Form \\
		\hline 
		real-part				&
		$\begin{array} {ll}
		= & \mathbf{x}-\tfrac{1}{3!}\mathbf{x}^3+\tfrac{1}{\textbf{5}!}\mathbf{x}^\textbf{5}-\tfrac{1}{\textbf{7}!}\mathbf{x}^\textbf{7}+\dots \\
		= & \sum_{n=0}\tfrac{-1^n}{(2n+1)!}\text{ }\mathbf{x}^{2n+1}
		\end{array}$
		& $=\sin(\mathbf{x})$\\ \hline 
		$\mathbf{\epsilon}$-part			&
		$\begin{array} {ll}
		= & a\mathbf{\epsilon}-\tfrac{1}{2!}\mathbf{x}^2a\mathbf{\epsilon}+\tfrac{1}{4!}a\mathbf{\epsilon} \mathbf{x}^4-\tfrac{1}{6!}a\mathbf{\epsilon} \mathbf{x}^6+\dots \\
		= & \mathbf{\epsilon}(a-\tfrac{1}{2!}a\mathbf{x}^2+\tfrac{1}{4!}a\mathbf{x}^4-\tfrac{1}{6!}a\mathbf{x}^6+\dots) \\
		= & \mathbf{\epsilon}\sum_{n=0}\tfrac{-1^n}{(2n)!}\text{ }a\mathbf{x}^{2n}
		\end{array}$ 
		& $=\mathbf{\epsilon}(a\cos(\mathbf{x}))$ \\ \hline 		
		$\mathbf{\epsilon}^2$-part			&
		$\begin{array} {ll}
		= & -\tfrac{3}{3!}a^2\mathbf{\epsilon}^2\mathbf{x}+\tfrac{\textbf{10}}{\textbf{5}!}a^2\mathbf{\epsilon}^2\mathbf{x}^3-\tfrac{21}{\textbf{7}!}a^2\mathbf{\epsilon}^2\mathbf{x}^\textbf{5}+\dots \\
		= &
		-\tfrac{1}{2!}a^2\mathbf{\epsilon}^2\mathbf{x}+\tfrac{2}{4!}a^2\mathbf{\epsilon}^2\mathbf{x}^3-\tfrac{3}{6!}a^2\mathbf{\epsilon}^2\mathbf{x}^\textbf{5}+\dots \\
		= & \mathbf{\epsilon}^2(-\tfrac{1}{2!}\mathbf{x}+\tfrac{2}{4!}a^2\mathbf{x}^3-\tfrac{3}{6!}a^2\mathbf{x}^\textbf{5}+\dots) \\
		= & \mathbf{\epsilon}^2\sum_{n=0}\tfrac{-1^{n+1}(n+1)}{(2n+2)!}\text{ }a^2\mathbf{x}^{2n+1}
		\end{array}$ 
		& $=\mathbf{\epsilon}^2(-\tfrac{a^2}{2}\sin(\mathbf{x}))$\\ \hline 		
		$\mathbf{\epsilon}^3$-part			&
		$\begin{array} {ll}
		= & -\tfrac{1}{3!}a^3\mathbf{\epsilon}^3+\tfrac{\textbf{10}}{\textbf{5}!}a^3\mathbf{\epsilon}^3\mathbf{x}^2-\tfrac{\textbf{3\textbf{5}}}{\textbf{7}!}a^3\mathbf{\epsilon}^3\mathbf{x}^4+\tfrac{\textbf{84}}{\textbf{9}!}a^3\mathbf{\epsilon}^3\mathbf{x}^6-\dots \\
		= & \mathbf{\epsilon}^3(-\tfrac{1}{3!}a^3\mathbf{x}^0+\tfrac{\textbf{10}}{\textbf{5}!}a^3\mathbf{x}^2-\tfrac{\textbf{3\textbf{5}}}{\textbf{7}!}a^3\mathbf{x}^4+\tfrac{\textbf{84}}{\textbf{9}!}a^3\mathbf{x}^6-\dots) \\
		= & \mathbf{\epsilon}^3\sum_{n=0}\tfrac{-1^{n+1}}{6(2n)!}\text{ }a^3\mathbf{x}^{2n}
		\end{array}$ 
		& $=\mathbf{\epsilon}^3(-\tfrac{a^3}{6}\cos(\mathbf{x}))$\\ \hline 		
		$\mathbf{\epsilon}^4$-part			&
		$\begin{array} {ll}
		= & \tfrac{\textbf{5}}{\textbf{5}!}a^4\mathbf{\epsilon}^4\mathbf{x}-\tfrac{\textbf{3\textbf{5}}}{\textbf{7}!}a^4\mathbf{\epsilon}^4\mathbf{x}^3+\tfrac{\textbf{126}}{\textbf{9}!}a^4\mathbf{\epsilon}^4\mathbf{x}^\textbf{5}-\tfrac{\textbf{3\textbf{30}}}{11!}a^4\mathbf{\epsilon}^4\mathbf{x}^\textbf{7}+\dots \\
		= & \mathbf{\epsilon}^4(\tfrac{1}{4!}\mathbf{x}-\tfrac{\textbf{5}}{6!}a^4\mathbf{x}^3+\tfrac{\textbf{14}}{8!}a^4\mathbf{x}^\textbf{5}-\tfrac{\textbf{30}}{\textbf{10}!}a^4\mathbf{x}^\textbf{7}+\dots) \\
		= & \mathbf{\epsilon}^4\sum_{n=0}\tfrac{-1^n}{24(2n+1)!}\text{ }a^4\mathbf{x}^{2n+1}
		\end{array}$ 
		& $=\mathbf{\epsilon}^4(\tfrac{a^4}{24}\sin(\mathbf{x}))$\\ \hline 
	\end{tabular}
	\label{table:sinx+epsilon}
\end{table}
\end{example}

\begin{example}
	\label{ex:cosx+epsilon}
	Suppose that we have $\mathbf{x}=x+a\epsilon$ where $x,a\in\R$. Here we try to find what $\cos\mathbf{x}$ is. With a similar method to the one used in \Cref{ex:sinx+epsilon}, we have a calculation like what is shown in \Cref{table:cosx+epsilon}.
	\begin{table} 
		\caption{$\texttt{St}(\cos(\mathbf{x)})$ and the first two members of $\texttt{Nst}_\epsilon(\cos(\mathbf{x}))$}
		\centering 
		\begin{tabular}{|l|l|l|}
			\hline 
			real-part				&
			$\begin{array} {ll}
			= & 1-\tfrac{1}{2!}\mathbf{x}^2+\tfrac{1}{\textbf{4}!}\mathbf{x}^\textbf{4}-\tfrac{1}{\textbf{6}!}\mathbf{x}^\textbf{6}+\dots \\
			= & \sum_{n=0}\tfrac{-1^n}{(2n)!}\text{ }\mathbf{x}^{2n}
			\end{array}$
			& $=\cos(\mathbf{x})$\\ \hline 
			
			$\mathbf{\epsilon}$-part			&
			$\begin{array} {ll}
			= & -\tfrac{2}{2!}a\mathbf{\epsilon} \mathbf{x}+\tfrac{\textbf{4}}{\textbf{4}!}a\mathbf{\epsilon} \mathbf{x}^3-\tfrac{\textbf{6}}{\textbf{6}!}a\mathbf{\epsilon} \mathbf{x}^5+\dots \\
			= & \mathbf{\epsilon}(-a\mathbf{x}+\tfrac{1}{3!}a\mathbf{x}^3-\tfrac{1}{5!}a\mathbf{x}^5+\dots) \\
			= & \mathbf{\epsilon}\sum_{n=0} \tfrac{-1^{n+1}}{(2n+1)!}\text{ }a\mathbf{x}^{2n+1}
			\end{array}$ 
			& $=\mathbf{\epsilon}(-a\sin(\mathbf{x}))$ \\ \hline 
			
			$\mathbf{\epsilon}^2$-part			&
			$\begin{array} {ll}
			= & -\tfrac{1}{2!}a^2\mathbf{\epsilon}^2+\tfrac{\textbf{6}}{\textbf{4}!}a^2\mathbf{\epsilon}^2\mathbf{x}^2-\tfrac{\textbf{15}}{\textbf{6}!}a^2\mathbf{\epsilon}^2\mathbf{x}^\textbf{4}+\tfrac{\textbf{28}}{8!}a^2\mathbf{\epsilon}^2\mathbf{x}^\textbf{6}-\dots \\
			= & \mathbf{\epsilon}^2\sum_{n=0}\tfrac{-1^{n+1}(n+1)(2n+1)}{(2n+2)!}\text{ }a^2\mathbf{x}^{2n}\\
			= & \mathbf{\epsilon}^2\sum_{n=0}\tfrac{1}{2}\tfrac{-1^{n+1}}{(2n)!}\text{ }\mathbf{x}^{2n}\\
			\end{array}$ 
			& $=\mathbf{\epsilon}^2(-\tfrac{a^2}{2}\cos(\mathbf{x}))$\\ \hline 
		\end{tabular}
		\label{table:cosx+epsilon}
	\end{table}
	
Thus from \Cref{table:cosx+epsilon}, we get:
\begin{align*}
	\cos(\mathbf{x})=\cos(x+a\mathbf{\epsilon})=\langle\widehat{\cos(x)},-a\sin(x),-\tfrac{a^2}{2!}\cos(x),\tfrac{a^3}{3!}\sin(x),\tfrac{a^4}{4!}\cos(x),\dots\rangle,
\end{align*}

and we also get
\begin{align*}
	\cos(\mathbf{\epsilon})=\langle\widehat{1},0,-\tfrac{1}{2},0,\dots\rangle=1-\tfrac{\textbf{1}}{\textbf{2}!}\mathbf{\epsilon}^\textbf{2}+\dots
	\end{align*}
	for an infinitesimal angle $\mathbf{\epsilon}$.
\end{example}

\begin{example}
	\label{ex:expx+epsilon}
	With the same $\mathbf{x}$ as in \Cref{ex:sinx+epsilon,ex:cosx+epsilon}, we try to know what $\exp(\mathbf{x})$ is. Based on \Cref{def:ExpAsSeries},
	\begin{align*}
	\exp(\mathbf{x})=\exp(x+a\mathbf{\epsilon})=\textbf{1}+(x+a\mathbf{\epsilon})-\tfrac{\textbf{1}}{\textbf{2}!}(x+a\mathbf{\epsilon})^\textbf{2}+\tfrac{\textbf{1}}{\textbf{3}!}(x+a\mathbf{\epsilon})^\textbf{3}+\dots
	\end{align*}
	\noindent Our task now is to find all the members of $\texttt{Nst}_\epsilon(\exp(\mathbf{x}))$ and also $\texttt{St}(\exp(\mathbf{x}))$, which are shown in \Cref{table:expx+epsilon}. Note that from the way we define the function $\exp$, $\forall x_i\in\texttt{Nst}_\omega(\exp(\mathbf{x}))$ $x_i=0$. Thus from \Cref{table:expx+epsilon}, we get:
	\begin{align*}
	\exp(\mathbf{x})=\exp(x+a\mathbf{\epsilon})=\langle\widehat{\exp(x)},a\exp(x),\tfrac{a^2}{2!}\exp(x),\tfrac{a^3}{3!}\exp(x),\tfrac{a^4}{4!}\exp(x),\dots\rangle,
	\end{align*}
	and we also get
	\begin{align*}
	\exp(\mathbf{\epsilon})=\langle\widehat{1},1,\tfrac{1}{2!},\tfrac{1}{3!},\tfrac{1}{4!},\dots\rangle=1+\mathbf{\epsilon}+\tfrac{\textbf{1}}{\textbf{2}!}\mathbf{\epsilon}^\textbf{2}+\tfrac{\textbf{1}}{\textbf{3}!}\mathbf{\epsilon}^\textbf{3}-\dots
	\end{align*}
	for an infinitesimal angle $\mathbf{\epsilon}$.
	
	\begin{table} 
		\caption{$\texttt{St}(\exp(\mathbf{x)})$ and The First Three Members of $\texttt{Nst}_\epsilon(\exp(\mathbf{x}))$}
		\centering 
		\begin{tabular}{|l|l|l|}
			\hline 
			& Expanded Form	& Simplified Form \\
			\hline 
			real-part				&
			$\begin{array} {ll}
			= & 1+\mathbf{x}+\tfrac{1}{2!}\mathbf{x}^2+\tfrac{1}{3!}\mathbf{x}^3+\dots \\
			= & \sum_{n=0}\tfrac{1}{n!}\text{ }\mathbf{x}^n
			\end{array}$
			& $=\exp(\mathbf{x})$\\ \hline 
			
			$\mathbf{\epsilon}$-part			&
			$\begin{array} {ll}
			= & a\mathbf{\epsilon}+\tfrac{2}{2!}a\mathbf{\epsilon} \mathbf{x}+\tfrac{3}{3!}a\mathbf{\epsilon} \mathbf{x}^2+\tfrac{\textbf{4}}{\textbf{4}!}a\mathbf{\epsilon} \mathbf{x}^3\mathbf{\epsilon}+\dots \\
			= & \mathbf{\epsilon}(a+a\mathbf{x}+\tfrac{1}{2!}a\mathbf{x}^2+\tfrac{1}{3!}a\mathbf{x}^3-\tfrac{1}{\textbf{4}!}a\mathbf{x}^\textbf{4}+\dots) \\
			= & \mathbf{\epsilon}\sum_{n=0}\tfrac{1}{n!}\text{ }a\mathbf{x}^n
			\end{array}$ 
			& $=\mathbf{\epsilon}(a\exp(\mathbf{x}))$ \\ \hline 
			
			$\mathbf{\epsilon}^2$-part			&
			$\begin{array} {ll}
			= & \tfrac{1}{2!}a^2\mathbf{\epsilon}^2\mathbf{x}+\tfrac{3}{3!}a^2\mathbf{\epsilon}^2\mathbf{x}+\tfrac{6}{\textbf{4}!}a^2\mathbf{\epsilon}^2\mathbf{x}^2+\tfrac{\textbf{10}}{5!}a^2\mathbf{\epsilon}^2\mathbf{x}^3+\dots \\
			= &
			\mathbf{\epsilon}^2(\tfrac{1}{2!}a^2+\tfrac{3}{3!}a^2\mathbf{x}+\tfrac{6}{\textbf{4}!}a^2\mathbf{x}^2+\tfrac{\textbf{10}}{5!}a^2\mathbf{x}^3+\dots) \\
			= & \mathbf{\epsilon}^2\sum_{n=0}\tfrac{(n+1)(n+2)}{2(n+2)!}\text{ }a^2\mathbf{x}^n\\
			= & \mathbf{\epsilon}^2\sum_{n=0}\tfrac{1}{2(n!)}\text{ }a^2\mathbf{x}^n
			\end{array}$ 
			& $=\mathbf{\epsilon}^2(\tfrac{a^2}{2}\exp(\mathbf{x}))$\\ \hline 
			
			$\mathbf{\epsilon}^3$-part			&
			$\begin{array} {ll}
			= & \tfrac{1}{3!}a^3\mathbf{\epsilon}^3+\tfrac{\textbf{4}}{\textbf{4}!}a^3\mathbf{\epsilon}^3\mathbf{x}+\tfrac{\textbf{10}}{5!}a^3\mathbf{\epsilon}^3\mathbf{x}^2+\tfrac{\textbf{20}}{6!}a^3\mathbf{\epsilon}^3\mathbf{x}^3+\tfrac{\textbf{35}}{7!}a^3\mathbf{\epsilon}^3\mathbf{x}^\textbf{4}+\dots \\
			= & \mathbf{\epsilon}^3(\tfrac{1}{3!}a^3+\tfrac{\textbf{4}}{\textbf{4}!}a^3\mathbf{x}-\tfrac{\textbf{10}}{5!}a^3\mathbf{x}^2+\tfrac{\textbf{20}}{6!}a^3\mathbf{x}^3+\tfrac{\textbf{35}}{7!}a^3\mathbf{x}^\textbf{4}+\dots) \\
			= & \mathbf{\epsilon}^3\sum_{n=0}\tfrac{(n+1)(n+2)(n+3)}{6(n+3)!}\text{ }a^3\mathbf{x}^n\\
			= & \mathbf{\epsilon}^3\sum_{n=0}\tfrac{1}{6(n!)}\text{ }a^3\mathbf{x}^n
			\end{array}$ 
			& $=\mathbf{\epsilon}^3(\tfrac{a^3}{6}\exp(\mathbf{x}))$\\ \hline 
		\end{tabular}
		\label{table:expx+epsilon}
	\end{table}
\end{example}

From the preceding discussion, we have the following proposition.

\begin{proposition} For the $\sin$, $\cos$, and $\exp$ functions:
	\begin{enumerate}
		\item $\textnormal{\texttt{Der}}_{\sin{\mathbf{x}}}(\mathbf{x},\mathbf{\epsilon})=\tfrac{\sin(\mathbf{x}+\mathbf{\epsilon})-\sin(\mathbf{x})}{\mathbf{\epsilon}}=\langle\widehat{\cos(x)},-\tfrac{1}{2!}\sin(x),-\tfrac{1}{3!}\cos(x),\dots\rangle$, and so we have: 
		\begin{center}
			$\textnormal{\texttt{St}}(\textnormal{\texttt{Der}}_{\sin(\mathbf{x})}(\mathbf{x},\mathbf{\epsilon}))=\cos(\mathbf{x})$
		\end{center}
		\item $\textnormal{\texttt{Der}}_{\cos{\mathbf{x}}}(\mathbf{x},\mathbf{\epsilon})=\tfrac{\cos(\mathbf{x}+\mathbf{\epsilon})-\cos(\mathbf{x})}{\mathbf{\epsilon}}=\langle\widehat{-\sin(x)},-\tfrac{1}{2!}\cos(x),\tfrac{1}{3!}\sin(x),\dots\rangle$, and so 
		\begin{center}
			$\textnormal{\texttt{St}}(\textnormal{\texttt{Der}}_{\cos(\mathbf{x})}(\mathbf{x},\mathbf{\epsilon}))=-\sin(\mathbf{x})$
		\end{center}
		\item $\textnormal{\texttt{Der}}_{\exp{\mathbf{x}}}(\mathbf{x},\mathbf{\epsilon})=\tfrac{\exp(\mathbf{x}+\mathbf{\epsilon})-\exp(\mathbf{x})}{\mathbf{\epsilon}}=\langle\widehat{\exp(x)},\tfrac{1}{2!}\exp(x),\tfrac{1}{3!}\exp(x),\dots\rangle$, and so 
		\begin{center}
			$\textnormal{\texttt{St}}(\textnormal{\texttt{Der}}_{\exp(\mathbf{x})}(\mathbf{x},\mathbf{\epsilon}))=\exp(\mathbf{x})$
		\end{center}
	\end{enumerate}
\end{proposition}

\subsection{Continuity}

In this subsection, we try to pinpoint what the good definition for continuous functions is. We also decide whether we can permeate it between $\R$ and $\Rzl$. Note that if the domain and codomain of a function is not explicitly stated, they will be determined from the specified model.

\begin{definition}[$\text{ED}_\text{CLASS}$]
	\label{def:ED_CLASS}
	A function $f:\R\rightarrow\R$ is continuous at a point $a\in\R$ if, given $n\in\mathbb{N}$, there exists a $m\in\mathbb{N}$ such that
	\begin{center}
		$|f(x)-f(a)|<\frac{1}{n}$ whenever $|x-a|<\frac{1}{m}$. 
	\end{center}
	The function $f$ is called continuous on an interval $I$ iff $f$ is continuous at every point in $I$.
\end{definition}

\begin{definition}[$\text{ED}$]
	\label{def:ED}
	A function $f$ is continuous at a point $\mathbf{c}\in\Rzl$ if, given $\mathbf{\einf_1}\in\Rzl>\textbf{0}$, there exists a $\einf_2\in\Rzl>\textbf{0}$ such that
	\begin{center}
	$|\mathbf{f(x)}-\mathbf{f(c)}|<\mathbf{\einf_1}$ whenever $|\mathbf{x}-\mathbf{c}|<\mathbf{\einf_2}$.
	\end{center}
	That function $\mathbf{f}$ is called continuous function over an interval $I$ iff $\mathbf{f}$ is continuous at every point in $I$.
\end{definition}

\begin{proposition}
	\label{prop:fedclassnoted}
	There exists a function $f$ in $\Rzl$ which is continuous under \Cref{def:ED}, but discontinuous under \Cref{def:ED_CLASS}, i.e.
	\begin{center}
		$\Rzl\models\exists f$ s.t. $(\textnormal{ED}_\textnormal{CLASS}(f)\land\neg\textnormal{ED}(f))$.
	\end{center}
\end{proposition}
\begin{proof}
	Suppose that $\Delta=\{x\mid\forall n\in\mathbb{N},\abs{x}<\frac{1}{n}\}$ -- in other words, $\Delta$ is a set of all infinitesimals -- and consider the indicator function around $\Delta$, that is
	\begin{align*}
		\mathds{1}_\Delta(x)=
		\begin{cases}
			1, & x\in\Delta\\
			0, & \textnormal{otherwise.}
		\end{cases}
	\end{align*}
	Then $\textnormal{ED}_\textnormal{CLASS}(\mathds{1}_\Delta)$ but $\neg\textnormal{ED}(\mathds{1}_\Delta)$.
\end{proof}

\begin{remark}
	Note that:
	\begin{enumerate}
		\item The set $\Delta$ in $\R$ only has 0 as its member. That is $\R\models\Delta=\{0\}$.
		\item In $\R$, both Definitions \ref{def:ED_CLASS} and \ref{def:ED} are equivalent, that is for any function $f$, $\R\models\textnormal{ED}_\textnormal{CLASS}(f)\leftrightarrow\textnormal{ED}(f)$.
	\end{enumerate}
\end{remark}


\begin{property}[$\text{EVP}$]
	\label{prop:EVP}
	If $I$ is an interval and $f:I\rightarrow J$, we say that $f$ has the extreme value property iff $f$ has its maximum value on $I$. That is,
	\begin{center}
		$\forall a\leq b\in I$, $\exists x\in[a,b]$ s.t. $\forall y\in[a,b](f(y)\leq f(x))$.
	\end{center}
\end{property}


\begin{property}[$\text{IVP}$]
	\label{prop:IVP}
	If $I$ is an interval, and $f:I=[a,b]\rightarrow J$, we say that $f$ has the intermediate value property iff
	\begin{center}
		$\forall c'\in(f(a),f(b))$, $\exists c\in(a,b)$ s.t. $f(c)=c'$.
	\end{center} 
\end{property}


\begin{theorem}
	$\R\models\textnormal{ED}\rightarrow\textnormal{EVP}$
\end{theorem}

\begin{proof}
	The proof of this theorem can be found in any standard book for Analysis course (in \cite{bartle1992introduction} for example).
\end{proof}

\begin{theorem}
	\label{thm:edRZ<NotImpliesEVPRZ<}
	There is a function $f$ such that
	\begin{center}
		$\Rzl\models\textnormal{ED}(f)\land\lnot\textnormal{EVP}(f)$.
	\end{center}
\end{theorem}

\begin{proof}
	Take the function $f$ on $[1,2]$ as defined below:	
	\begin{align*}
	f(x)=
	\begin{cases}
	\frac{1}{n} & x\sim\frac{m}{n}\textnormal{ (reduced fraction)} \\
	0 & \text{otherwise}.
	\end{cases}
	\end{align*}
\end{proof}

\begin{remark}
	This research now reaches an especially engrossing object. The function $f(x)$ in \Cref{thm:edRZ<NotImpliesEVPRZ<} can be used to construct a fractal-like object. Fractals are classically defined as geometric objects that exhibit some form of self-similarity. Figure \ref{fig:zoom} shows what the function $f(x)$ in \Cref{thm:edRZ<NotImpliesEVPRZ<} looks like, and also what occurs when we zoom in on a particular point. In this sense, the function from Theorem \ref{thm:edRZ<NotImpliesEVPRZ<} is an  \textit{infinitesimal fractal}. Formally speaking, suppose that we have a function $f:\Rzl\rightarrow\R$ and let us define another function
	
\begin{align*}
	F:\Rzl\rightarrow\Rzl
\end{align*}
by
\begin{align*}
	\begin{aligned}
		F(x) &= f(x)+\epsilon f(x)+\epsilon^2 f(x)+\dots  \\
		&= \langle \widehat{f(x)},f(x),f(x),\dots\rangle.
	\end{aligned}
\end{align*}
Then, that function $F(x)$ will define an \textit{infinite fractal} (if $\texttt{ni}_\omega(\frac{1}{\epsilon}\widehat{\times}F(x))=F(x)$) or \textit{infinitesimal fractals} (if $\texttt{ni}_\omega(\epsilon\widehat{\times}F(x))=F(x)$), where $\texttt{ni}_\omega(\mathbf{x})$ denotes the non-infinity part of $\mathbf{x}$.

\begin{figure}
	\includegraphics[width=0.3\textwidth]{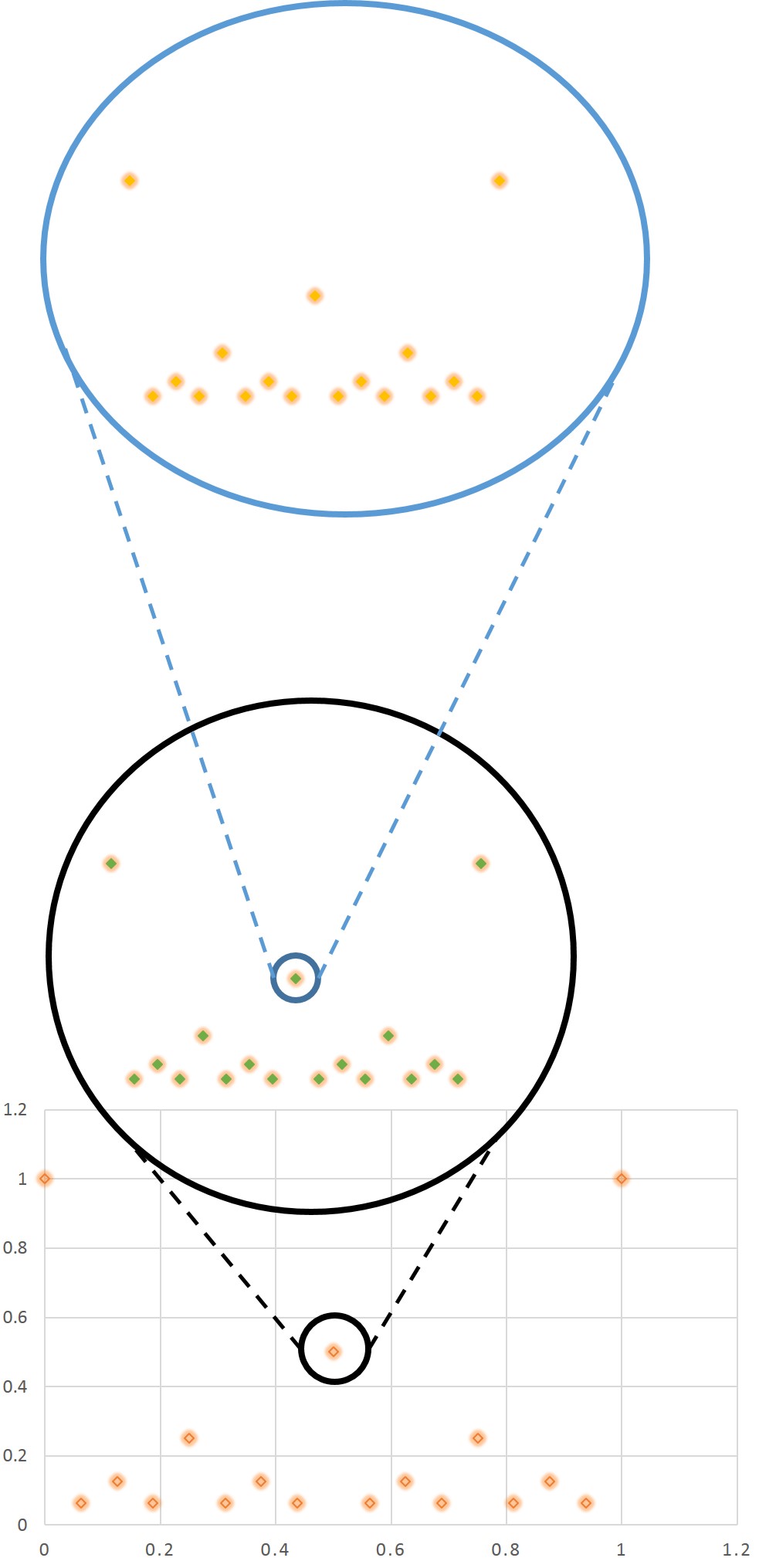}
	\caption{Illustration of the infinitesimal fractal from the function defined in Theorem \ref{thm:edRZ<NotImpliesEVPRZ<}}
	\label{fig:zoom}
\end{figure}
\end{remark}

Figure \ref{fig:RelationshipAmongThreeDefinitionsofContinuity} shows the relationship among the three definitions of continuity in both $\R$ and $\Rzl$. The proof of each of the relations there can be seen in \cite[p.~94--98]{nugraha2018naive}.

\begin{figure}[h]
	\subfloat[In $\R$]{
		\begin{tikzpicture}[scale=0.8, every node/.style={transform shape}]
			\path       node (EVP) {$\textnormal{EVP}$}
			(0:4cm)   node (IVP) {$\textnormal{IVP}$}
			(-60:4cm) node (ED) {$\textnormal{ED}$};
			\path[-stealth]
			(EVP.315) edge [strike through] (ED.105)
			(ED.135) edge (EVP.285)
			(IVP.255) edge [strike through] (ED.45)
			(ED.75)  edge (IVP.225) 
			(EVP.350) edge [strike through] (IVP.190)
			(IVP.165) edge [strike through] (EVP.15);
		\end{tikzpicture}
	}
	\subfloat[In $\Rzl$]{
		\begin{tikzpicture}[scale=0.8, every node/.style={transform shape}]
		\path       node (EVPRZ) {$\textnormal{EVP}$}
		(0:4cm)   node (IVPRZ) {$\textnormal{IVP}$}
		(-60:4cm) node (EDRZ) {$\textnormal{ED}$};
		\path[-stealth]
		(EVPRZ.315) edge [strike through] (EDRZ.105)
		(EDRZ.135) edge [strike through] (EVPRZ.285)
		(IVPRZ.255) edge [strike through] (EDRZ.45)
		(EDRZ.75)  edge [strike through] (IVPRZ.225) 
		(EVPRZ.350) edge [strike through] (IVPRZ.190)
		(IVPRZ.165) edge [strike through] (EVPRZ.15);
		\end{tikzpicture}
	}
	\caption{Relationship among the three definitions of continuity}
	\label{fig:RelationshipAmongThreeDefinitionsofContinuity}
\end{figure}
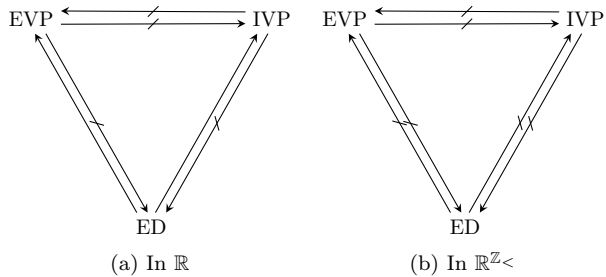

The obvious question worth asking is how do we define continuity in our set $\Rzl$. As seen before, there are three possible ways to define it, namely: with the $\epsilon$-$\delta$ definition (ED), with extreme value property (EVP), or with the intermediate value property (IVP). We will now discuss them one by one. \textit{Firstly}, through IVP. The IVP basically says that for every value within the range of the given function, we can find a point in the domain corresponding to that value. Will this work in our set $\Rzl$? Let us consider the $\Rzl$-valued function $f(x)=x^2$ on $[a,b]$ for any $a,b\in\Rzl$ and let us assume that IVP holds. It follows that for every $c'$ between $f(a)=a^2$ and $f(b)=b^2$, $\exists c\in(a,b)$ such that $f(c)=c^2=c'$. The only $c$ which satisfies that last equation is $c=\sqrt{c'}$, which cannot be defined in our set $\Rzl$. Thus, IVP, even though it is somehow intuitively ``obvious'', it does not work in $\Rzl$. This phenomenon is actually not uncommon if we want to have a world with infinitesimals (or infinities) in it. See \cite[p.~107]{bell1998primer} for example.

However, note that in $\R$, the function $x^2$ still satisfies IVP. Now, is there a function in $\Rzl$ that satisfies IVP? Consider the identity function $f(x)=x$. This function clearly satisfies IVP in both domains, and so we have the following theorem.

\begin{theorem}
	There exists a function $f$ such that $(\R\models\textnormal{IVP}(f)\land\Rzl\models\lnot\textnormal{IVP}(f))$, and there exists a function $g$ s.t. $\R,\Rzl\models\textnormal{IVP}(g)$.
\end{theorem}

\noindent Hence, from the argument above we also argued that defining continuity in our set with IVP is not really useful.

Secondly, in regards to EVP. This is clearly not a good way to define continuity in our set because even in the set of real numbers, there are some continuous functions which do not satisfy EVP themselves. So the last available option now is the third one, which is the $\epsilon$-$\delta$ (ED) definition. We argued that this definition is the best way to define continuity in $\Rzl$. Moreover, in this way, it preserves much of the spirit of classical analysis on $\R$ while retaining the intuition of infinitesimals.

It is important to note that in the ED definition of continuity (Definition \ref{def:ED}), there are two variables which are in play, i.e. $\einf_1$ and $\einf_2$. When we applied this definition on our set, these two variables hold important (or rather, very interesting) roles where we will have different levels of continuity from the same function. What we mean is that these two variables can greatly vary depending on how far (`deep') we want to push (observe) them, e.g. $\einf_2$ can be a real number ($\einf_2\in\Delta^0$), or it can be in $\Delta^4$, $\Delta^8$ and so on. Remember that these two numbers, $\einf_1$ and $\einf_2$, will determine how subtle we want our intervals to be (see Figure \ref{fig:Illustrationofeinf1andeinf2} for illustration).

\begin{figure}
	\subfloat[An $\einf_1$ bound \& its $\einf_2$ neighbourhood fulfilling Definition \ref{def:ED}]{
		\includegraphics[width=.4\linewidth]{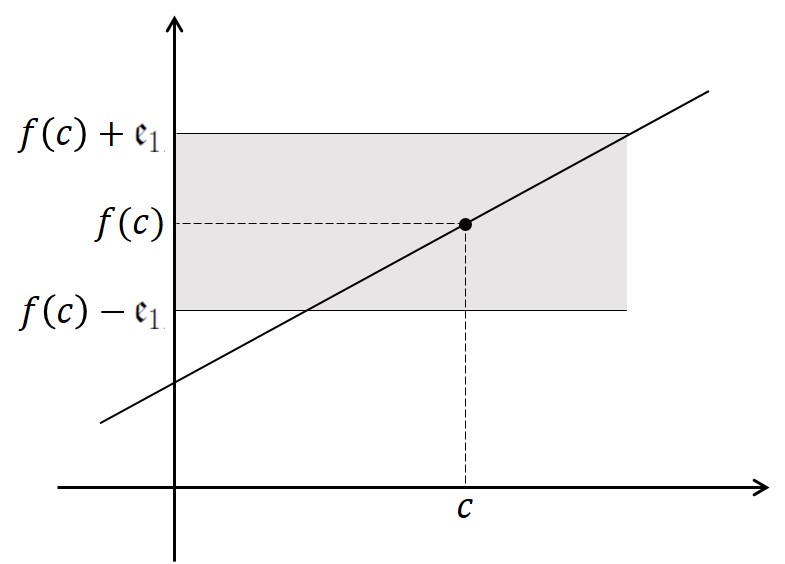}\hfill
		\includegraphics[width=.4\linewidth]{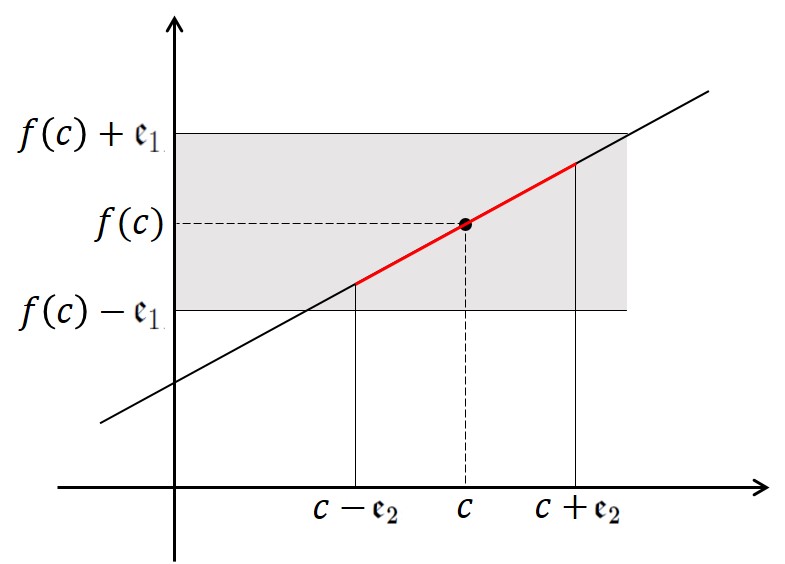}\hfill
	}
	\par\medskip
	\subfloat[A smaller bound \& its neighbourhood]{
			\includegraphics[width=.45\linewidth]{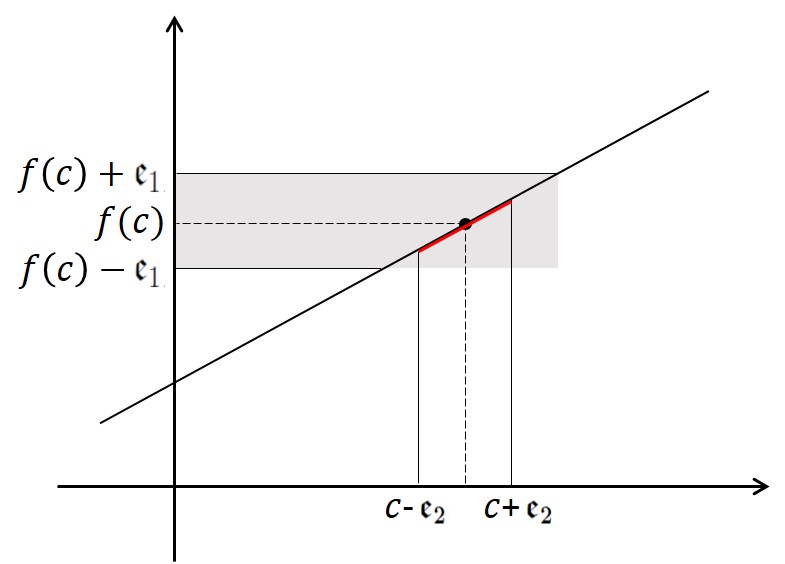}\hfill
	}
	\par\medskip
	\caption{Illustration of $\einf_1$ and $\einf_2$ intervals}
	\label{fig:Illustrationofeinf1andeinf2}
\end{figure}

%

Thus this definition of continuity works as follows. Suppose that we have a function $f$ and we want to decide whether it is continuous or not. With this concept of two variables, we will have what we call as $(k,n)$-continuity where $k,n\in\mathbb{N}\cup\{0\}$.

\begin{definition}[$(k,n)$-Continuity]
	\label{def:k-ncontinuity}
	A function $f$ is $(k,n)$-continuous at a point $\mathbf{c}$ iff $\forall\mathbf{\einf_{1_k}}>0$, $\exists\einf_{2_n}>0$ such that
	\begin{center}
		if $|\mathbf{x}\widehat{-}\mathbf{c}|<\mathbf{\einf_{2_n}}$, then $|\mathbf{f(x)}\widehat{-}\mathbf{f(c)}|<\mathbf{\einf_{1_k}}$
	\end{center}
	where $\einf_{1_p},\einf_{2_p}\in\Delta^p$.
\end{definition}

\begin{definition}
	A function $f$ is said to be $(k,n)$-continuous iff it is $(k,n)$-continuous at every point in the given domain.
\end{definition}

\begin{remark}
	\label{remark:kncontinuityintervalsubset}
	From the definition of the set $\Delta^m$, note that for any $r\in\Rzl,$ $d\in\Delta^p,$ and $e\in\Delta^{p+1},$
	\begin{align*}
		(r-e,r+e)\subseteq(r-d,r+d).
	\end{align*}
\end{remark}

To be able to grasp a better understanding of Definition \ref{def:k-ncontinuity}, see the examples below.

\begin{example}
	\label{ex:xandx+1fork-ncontinuity}
	Consider the $\Rzl$-valued function $f(x)$ defined as follows:
	\[
	f(\mathbf{x})=
	\begin{cases} 
		\mathbf{x} & \texttt{St}(\mathbf{x})\leq 1 \\
		\mathbf{x}\widehat{+}\mathbf{1} & \textnormal{otherwise.} 
	\end{cases}
	\]
	First we need to understand clearly how this function actually works. Figure \ref{fig:xandx+1fork-ncontinuity}, where $i$ denotes an arbitrary infinitesimal number, illustrates to us what the function $f(\mathbf{x})$ looks like. Notice that at $\mathbf{x}=\mathbf{1}$, what looks like a point in real numbers is actually a (constant) line when we zoom in deep enough into $\Rzl$\footnote{We have to be really careful here because if the first condition there was $\mathbf{x}\leq1$ (instead of $\texttt{St}(\mathbf{x})\leq1$), then there would be no line there --- it will be \emph{exactly} one point.}. So how about the continuity of this function? It is obvious that $f(\mathbf{x})$ is not $0,0$-continuous (by taking, for example, $\einf_{1_0}=\frac{1}{2}$ and $\mathbf{x}=\mathbf{1.5}$). However, interestingly enough, it is $\mathit{(0,1)}$-continuous by taking $\einf_{2_1}\in\Delta^1$. Why was that? The fact that $\einf_{2_1}\in\Delta^1$ and that it has to depend on $\einf_{1_0}$ means that $\einf_{1_0}$ has to be in $\Delta^1$ as well. Now, assigning $\einf_{2_1}=\einf_{1_0}$ is sufficient to prove its $\mathit{0,1}$-continuity. 
	\begin{figure}
			\includegraphics[width=.4\linewidth]{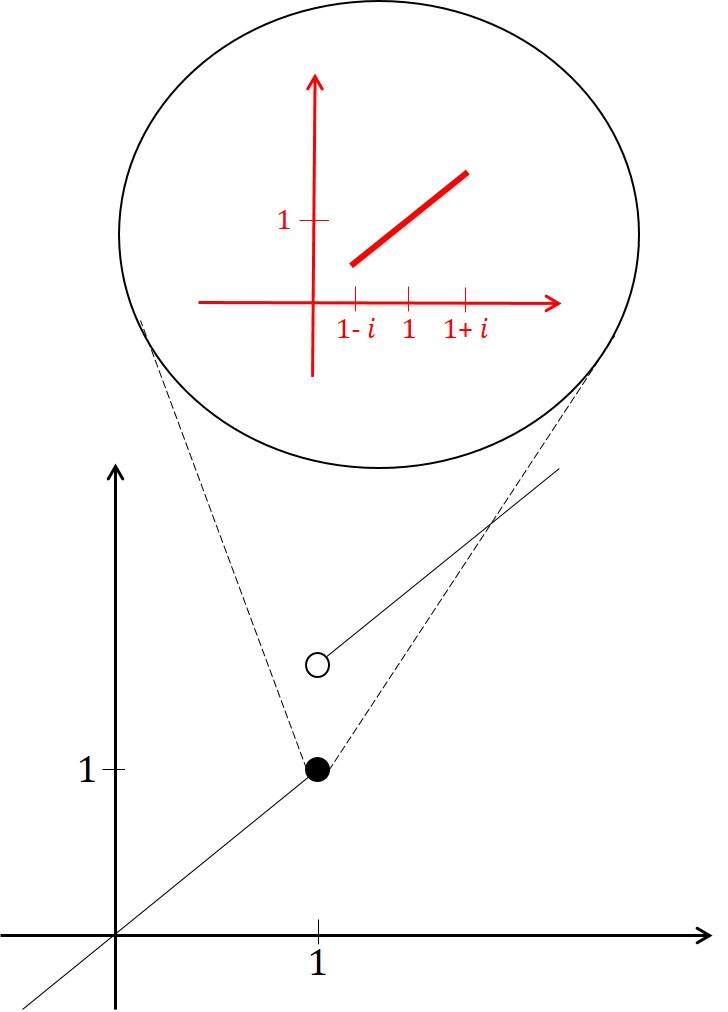}
			\caption{Illustration of Function $f(x)$ in Example \ref{ex:xandx+1fork-ncontinuity}}
			\label{fig:xandx+1fork-ncontinuity}
	\end{figure}
\end{example}

\begin{example}
	The identity function $f(\mathbf{x})=\mathbf{x}$ for all $\mathbf{x}\in\Rzl$ is $\mathit{(0,0)}$-continuous, just like in reals. However, it is not $\mathit{(1,0)}$-continuous because for any point $\mathbf{c}$, there is an $\einf_{1_1}=\mathbf{\epsilon}$ such that for every $\einf_{2_0}=r$ where $r\in\R$, $|\mathbf{x}\widehat{-}\mathbf{c}|<r$ but $f(\mathbf{x})-f(c)\geq\mathbf{\epsilon}$. In fact, identity function is $(k,n)$-continuous only when $k\leq n$, but not otherwise.
\end{example}

The next theorem below is very interesting in as much as it enables us to classify whether a function is a constant function or not by using $(k,n)$-continuity.

\begin{theorem}
	\label{thm:constantandkncont}
	For any function $f$, if $f$ is $(k,n)$-continuous for any $k,n\in\mathbb{N}\cup\{0\}$, then $f$ is a constant function.
\end{theorem}
\begin{proof}
	Here we want to prove its contrapositive, in other words, if $f$ is not constant, then there exist $(k,n)$ such that $f$ is not $(k,n)$-continuous. Because $f$ is not constant, there will be $a,b$ in the domain such that $f(a)\neq f(b)$ and suppose that $|f(a)-f(b)|\in\Delta^m$ such that $|a-b|\in\Delta^l$. By this construction, $f$ will not be $(m,l-1)$-continuous and so we can take $k=m$ and $n=l-1$.
\end{proof}

The theorem below is a generalisation of Theorem \ref{thm:constantandkncont}.

\begin{theorem}
	If there exists $m$ for all $k$ such that a function $f$ is $(k,m)$-continuous, then $f$ will be constant in $\Delta^m$-neighbourhood.
\end{theorem}


The next interesting question is: what is the relation between, for example, $\mathit{(0,1)}$-continuity and $\mathit{(0,2)}$-continuity? In general, what is the relation between $(k,n)$-continuity and $(k,(n+1))$-continuity? And also between $(k,n)$-continuity and $((k+1),n)$-continuity? See these two theorems below.

\begin{theorem}
	\label{thm:knandkn+1}
	For any function $f$, if $f$ is $(k,n)$-continuous, then $f$ is also $(k,(n+1))$-continuous.
\end{theorem}

\begin{proof}
	Suppose that a function $f$ is $(k,n)$-continuous at point $\mathbf{c}$. This would mean that $\forall\einf_{1_k}$, $\exists\einf_{2_n}$ such that if $|\mathbf{x}\widehat{-}\mathbf{c}|<\einf_{2_n}$, then $|f(\mathbf{x})\widehat{-}f(\mathbf{c})|<\einf_{1_k}$. By using the same $\einf_{1_k}$ and from Remark \ref{remark:kncontinuityintervalsubset}, we can surely find $\einf_{2_{(n+1)}}=\einf_{2_n}\widehat{\times}\mathbf{\epsilon}$ such that for all $x\in(\mathbf{c}\widehat{-}\einf_{2_{(n+1)}},\mathbf{c}\widehat{+}\einf_{2_{(n+1)}})$, $f(x)\in(f(\mathbf{c})\widehat{-}\einf_{1_k},f(\mathbf{c})\widehat{+}\einf_{1_k})$.
\end{proof}

\begin{example}
	By Theorem \ref{thm:knandkn+1}, the function $f(x)$ in Example \ref{ex:xandx+1fork-ncontinuity} is also $\mathit{(0,2)}$-continuous, and also $\mathit{(0,3)}$-continuous, and so on.
\end{example}

\begin{theorem}
	For any function $f$, if $f$ is $((k+1),n)$-continuous, then $f$ is also $(k,n)$-continuous.
\end{theorem}

\begin{proof}
	Suppose that a function $f$ is $((k+1),n)$-continuous at point $\mathbf{c}$ and the set $\Delta^m$ defined as in Theorem \ref{thm:knandkn+1}. The fact that $f$ is $((k+1),n)$-continuous means that $\forall\einf_{1_{(k+1)}}$, $\exists\einf_{2_n}$ such that if $|\mathbf{x}\widehat{-}\mathbf{c}|<\einf_{2_n}$, then $|f(\mathbf{x})\widehat{-}f(\mathbf{c})|<\einf_{1_{(k+1)}}$ is hold. Here we want to prove that $\forall\einf_{1_{k}}$, $\exists\einf_{2_n}$ such that if $|\mathbf{x}\widehat{-}\mathbf{c}|<\einf_{2_n}$, then $|f(\mathbf{x})\widehat{-}f(\mathbf{c})|<\einf_{1_k}$. This actually follows directly from Remark \ref{remark:kncontinuityintervalsubset} as $(f(\mathbf{c})\widehat{-}\einf_{1_{(k+1)}},f(\mathbf{c})\widehat{+}\einf_{1_{(k+1)}})\subseteq(f(\mathbf{c})\widehat{-}\einf_{1_{k}},f(\mathbf{c})\widehat{+}\einf_{1_k})$.
\end{proof}

Now suppose that $f$ and $g$ are two $(k,n)$-continuous functions in $\Rzl$. We will examine how the arithmetic of those two continuous functions works. It is clear that $(k,n)$-continuity is closed under addition and subtraction, i.e. $f+g$ and $f-g$ are both $(k,n)$-continuous. The composition and multiplication of two continuous functions are particularly interesting as can be seen in \Cref{thm:compositionContinuous} and \Cref{thm:multContinuous}, respectively.

\begin{theorem}
	\label{thm:compositionContinuous}
	If $f$ is a $(k,n)$-continuous function and $g$ is an $(n,q)$-continuous function, then $f\circ g$ will be $(k,q)$-continuous.
\end{theorem}
\begin{proof}
	Since $f$ is $(k,n)$-continuous at $g(c)$, our definition of continuity tells us that for all $\einf_{1_k}>0$, there exists $\einf_{2_n}$ such that
	\begin{center}
		if $\abs{g(x)-g(a)}<\einf_{2_n}$, then $\abs{f(g(x))-f(g(a))}<\einf_{1_k}$.
	\end{center}
	Also since $g$ is $(n,q)$-continuous at $c$, there exists $\einf_{2_q}$ such that
	\begin{center}
	if $\abs{x-a}<\einf_{2_q}$, then $\abs{g(x)-g(a)}<\einf_{2_n}$.
	\end{center}
	I have taken $\einf_{1_n}=\einf_{2_n}$ here. Now this tells us that for all $\einf_{1_k}>0$, there exists $\einf_{2_q}$ (and an $\einf_{2_n}$) such that
	\begin{center}
		if $\abs{x-a}<\einf_{2_q}$, then $\abs{g(x)-g(a)}<\einf_{2_n}$ which implies that $\abs{f(g(x))-f(g(c))}<\einf_{1_k}$,
	\end{center}
	which is what we wanted to show.
\end{proof}

\begin{theorem}
	\label{thm:multContinuous}
	Suppose that $f,g$ are finite-valued functions. If $f$ is a $(k,n)$-continuous function and $g$ is an $(l,o)$-continuous function, then the function $H=f\cdot g$ will be $(\max\{k,l\},\min\{n,o\})$-continuous.
\end{theorem}
\begin{proof}
	Let $f,g$ be given such that $f$ is $(k,n)$-continuous and $g$ is $(l,o)$-continuous. Now let $H$ be defined by $H(x)=f(x)g(x)$ and so we want to show that $H$ is $(\max\{k,l\},\min\{n,o\})$-continuous, that is, for all $c\in\Rzl$, for every $\einf_{1_{\max\{k,l\}}}>0$, there exists $\einf_{2_{\min\{n,o\}}}>0$ such that for all $x\in\Rzl$ with $\abs{x-c}<\einf_{2_{\min\{n,o\}}}$, $\abs{H(x)-H(c)}<\einf_{1_{\max\{k,l\}}}$ holds.\\
	Now let  $c$ and $\einf_{1_{\max\{k,l\}}}$ be given and we choose $\einf_2$ such that $\einf_2\in\Delta^{\min\{n,o\}}$, i.e. $\einf_2=\einf_{2_{\min\{n,o\}}}$. Then for all $x\in\Rzl$ with $\abs{x-c}<\einf_{2_{\min\{n,o\}}}$,
	\begin{align}
		\abs{H(x)-H(c)}		&= \abs{f(x)g(x)-f(c)g(c)}\nonumber\\
									 &= \abs{f(x)g(x)-f(x)g(a)+f(x)g(a)-f(c)g(c)}\nonumber\\
									 &\leq\abs{f(x)g(x)-f(c)g(c)} + \abs{f(x)g(a)-f(c)g(c)}\nonumber\\
									 &= \abs{f(x)(g(x)-g(a))} + \abs{g(a)(f(x)f(c))}\nonumber\\
									 &< \abs{f(x)}\einf_{1_{l}} + \abs{g(a)}\einf_{1_{k}} \label{key}
	\end{align}
	Note that because $f$ and $g$ are limited-valued function, then $\abs{f(x)}\einf_{1_{l}}$ and $\abs{g(a)}\einf_{1_{k}}$ are still in $\Delta^l$ and $\Delta^k$, respectively. This means that the right side of Inequality \ref{key} will be in $\Delta^{\max\{k,l\}}$ and so $H(x)-H(c)<\einf_{1_{\max\{k,l\}}}$ is hold.
\end{proof}

It is worth pointing out here that the definition of $(k,n)$-continuity is a much more fine-grained notion than the classical continuity. This is self-explanatory by the use of those two variables $k$ and $n$ which makes us able to take much more infinitesimals --- in other words, we will be able to examine a far greater depth --- than in the classical definition. Furthermore, there might be some possible connections to one of the quantum phenomenons in physics: `action at a distance'. This concept is typically characterized in terms of some cause producing a spatially separated effect in the absence of any medium by which the causal interaction is transmitted \cite{french2005action} and closely connected to the question of what the deepest level of physical reality is \cite[pg. 168]{musser2015spooky}. Note that research on this phenomenon is still being conducted up until now, as can be seen for example  \cite{popkin2018einstein}, \cite{shalm2015strong} and \cite{weston2018heralded}.

\subsection{Topological Continuity}

\begin{definition}
	A function $f$ from a topological space $(X,\uptau_1)$ to a topological space $(Y,\uptau_2)$ is a function $f:X\rightarrow Y$.
\end{definition}

\noindent From now on, we will abbreviate this function notation by $f:X\rightarrow\ Y$ or simply $f$ every time the topologies in $X$ and $Y$ need not be explicitly mentioned. Also, $f^{-1}$ denotes the inverse image of $f$ as usual.

\begin{definition}[St-continuous]
	\label{def:StTopContinuous}
	A function $f:X\rightarrow Y$ between topological spaces is \textit{standard topologically continuous}, denoted by St-continuous, if
	\begin{center}
		$f^{-1}(U)\subseteq X$ is St-open whenever $U\subseteq Y$ is St-open.
	\end{center}
\end{definition}

\begin{definition}[$\einf$-continuous]
	\label{def:EinfTopContinuous}
	A function $f:X\rightarrow\ Y$ between topological spaces is \textit{infinitesimally topologically continuous}, denoted by $\einf$-continuous, if
	\begin{center}
		$f^{-1}(U)\subseteq X$ is $\einf$-open whenever $U\subseteq Y$ is $\einf$-open.
	\end{center}
\end{definition}

\begin{theorem}
	\label{thm:Stcontiffedclass}
	Suppose that $X,Y\subseteq\Rzl$. Under the metric \textnormal{$\dis$}, a function $f:X\rightarrow Y$ is St-continuous if and only if $f$ satisfies $\textnormal{ED}_\textnormal{CLASS}$ definition (Definition \ref{def:ED_CLASS}).
\end{theorem}
\begin{proof}
	We need to prove the implication both ways.
	\begin{enumerate}
		\item We want to prove that if $f$ is St-continuous, then $f$ satisfies ED$_\textnormal{CLASS}$. Suppose that $f$ is St-continuous and let $\mathbf{x_o}\in X$ and $n\in\mathbb{N}>0$. Then, the ball
		\begin{align*}
			B_{f(\mathbf{x_0})}(\nicefrac{1}{n})=\{\mathbf{y}\in Y \mid\dis(\mathbf{y},f(\mathbf{x_0}))<\nicefrac{1}{n}\}
		\end{align*}
		is open in $Y$, and hence $f^{-1}(B_{f(\mathbf{x_0})})$ is open in $X$. Since $\mathbf{x_0}\in f^{-1}(B_{f(\mathbf{x_0})})$, there exists some balls of radius $\nicefrac{1}{m}$ for some $m\in\mathbb{N}$ such that
		\begin{align*}
			B_{\mathbf{x_0}}(\nicefrac{1}{m})\subseteq f^{-1}(B_{f(\mathbf{x_0})}).
		\end{align*}
		This is exactly what the ED$_\textnormal{CLASS}$ says.
		\item We want to prove that if $f$ satisfies ED$_\textnormal{CLASS}$, then $f$ is St-continuous. Suppose that $f$ satisfies ED$_\textnormal{CLASS}$ and let $U\subseteq Y$ is open. By Definition \ref{def:St-open}, for all $\mathbf{y}\in U$ there exists some $d_y=\nicefrac{1}{n_y}$ where $n_y\in\mathbb{N}$ such that
		\begin{align*}
			B_\mathbf{y}(d_y)\subseteq U
		\end{align*}
		and in fact,
		\begin{align}
			\label{key1}
			U=\bigcup_{y\in U}B_{d_y}(\mathbf{y}).
		\end{align}
		Now we claim that $f^{-1}(U)$ is open in $X$ and suppose that $\mathbf{x_0}\in f^{-1}(U)$. Then $f(\mathbf{x_0})\in U$ and so from Equation \ref{key1}, $f(\mathbf{x_0})\in B_{d_{y_0}}(\mathbf{y_0})$ for some $\mathbf{y_0}\in U$ and $d_{y_0}=\nicefrac{1}{n_{y_0}}$ for some $n_{y_0}\in\mathbb{N}$, i.e. $\dis(f(\mathbf{x_0}),\mathbf{y_0})<d_{y_0}$. Now define
		\begin{align}
			\label{key2}
			e=d_{y_0}-\dis(f(\mathbf{x_0}),\mathbf{y_0})>0.
		\end{align}
		By Definition \ref{def:ED_CLASS}, there exists some $m\in\mathbb{N}$ such that
		\begin{align}
			\label{key3}
			\textnormal{if } \mathbf{x}\in X \textnormal{ and } \dis(\mathbf{x},\mathbf{x_0})<\nicefrac{1}{m}, \textnormal{ then } \dis(f(\mathbf{x}),f(\mathbf{x_0}))<e.
		\end{align}
		Now we claim that
		\begin{align}
			\label{key4}
			B_{\mathbf{x_0}}(\nicefrac{1}{m})\subseteq f^{-1}(U),
		\end{align}
		which will actually show that $f^{-1}(U)$ is indeed open. To this end, let $x\in B_{\mathbf{x_0}}(\nicefrac{1}{m})$, i.e. $\dis(\mathbf{x},\mathbf{x_0})<\nicefrac{1}{m}$. Then from (\ref{key3}), we have $\dis(f(\mathbf{x}),f(\mathbf{x_0}))<e$. Then, the triangle inequality and (\ref{key2}) imply that
		\begin{align*}
			\dis(f(\mathbf{x}),\mathbf{y_0})\leq\dis(f(\mathbf{x}),f(\mathbf{x_0}))+\dis(f(\mathbf{x_0}),\mathbf{y_0})<e+\diss(f(\mathbf{x_0}),\mathbf{y_0})=d_{y_0}.
		\end{align*}
		This means that $f(x)\in B_{\mathbf{y_0}}(d_{y_0})\subseteq U$, so that $x\in f^{-1}(U)$. Therefore, (\ref{key4}) holds, as claimed.
	\end{enumerate}
	And so from those two points above, we have proved what we want.
\end{proof}

\begin{theorem}
	\label{thm:etopcontIffED}
	Suppose that $X,Y\subseteq\Rzl$. Under the metric \textnormal{$\dis$}, a function $f:X\rightarrow Y$ is $\einf$-continuous if and only if $f$ satisfies $\textnormal{ED}$ definition.
\end{theorem}
\begin{proof}
	The proof of this theorem is similar with the one in Theorem \ref{thm:Stcontiffedclass} with some slight modifications in the distances (from $\nicefrac{1}{n}$ for some $n\in\mathbb{N}$ into $\mathbf{\einf}\in\Rzl$).
\end{proof}

\subsection{Convergence}

When we are talking about sequences, it is necessary to talk also about what it means when we say that a sequence is convergent to a particular number. This subsection presents not only some possible definitions that can be used to define convergence in $\Rzl$, but also the problems which occur when we apply them in $\Rzl$.

\begin{definition}[Classical Convergence]
	\label{def:CC}
	A sequence $s_n$ converges to $s$ iff,
	\begin{center}
		$\forall m\in\mathbb{N}$, $\exists N$ such that $\forall n>N$, $\abs{s_n-s}<\frac{1}{m}$.
	\end{center}
	We write $\textnormal{CC}(s_n,s)$ to denote that a sequence $s_n$ is classically convergent to $s$.
\end{definition}

\noindent \Cref{def:CC} above is the standard definition of how we define the notion of convergent classically. 


\begin{definition}[Hyperconvergence]
	\label{def:HC}
	A sequence $s_n$ converges to $s$ iff,
	\begin{center}
		$\forall r>0$, $\exists N$ such that $\forall n>N$, $\abs{s_n-s}<r$.
	\end{center}
	We write $\textnormal{HC}(s_n,s)$ to denote that a sequence $s_n$ is hyperconvergent to $s$. The interpretation of $r$ can be either in $\R$ or in $\Rzl$.
\end{definition}


\begin{example}
	Suppose that we have a sequence $s_n=\epsilon^n$ as follows:
	\begin{center}
		$S_1=\epsilon=\langle\widehat{0},1,0,\dots\rangle$\\
		$S_2=\epsilon^2=\langle\widehat{0},0,1,0,\dots\rangle$\\
		$S_3=\epsilon^3=\langle\widehat{0},0,0,1,0,\dots\rangle$\\
		\text{ }\vdots\\
	\end{center}
	This sequence $s_n$ will hyperconverge to $\langle \widehat{0},0,0,\dots\rangle$, i.e. $s_n$ satisfies HC$(s_n,\textbf{0})$.
\end{example}

\begin{theorem}
	For any sequence $s_n$, $\R\models\textnormal{CC}(s_n,s)\leftrightarrow\textnormal{HC}(s_n,s)$.
\end{theorem}
\begin{proof}
	The proof from HC to CC is obvious. Now suppose that a sequence $s_n$ satisfies CC$(s_n)$ and w.l.o.g. we assume that the $r$ in HC definition is between 0 and 1. From the Archimedean property of reals we know that for every $0<r<1$, we can find an $m\in\mathbb{N}$ such that $\frac{1}{m}<r$, and so because of CC$(s_n)$, we have $\abs{s_n-s}<\frac{1}{m}<r$. 
\end{proof}

\begin{theorem}
	For any sequence $s_n$ in $\Rzl$, \textnormal{HC}($s_n,s$) always implies \textnormal{CC}($s_n,s$). However, there exists a sequence $(t_n)$ such that
	\begin{center}
		$\Rzl\models\textnormal{CC}(t_n,s)\not\rightarrow\textnormal{HC}(t_n,s)$.
	\end{center}
\end{theorem}

\begin{proof}
	\begin{enumerate}
		\item To prove the first clause, suppose that a sequence $s_n$ satisfies HC$(s_n,s)$. This means that we are able to find a number $N$ such that $\forall n\geq N$, $\abs{s_n-s}<r$ for any $r\in\Rzl$ which includes infinitesimals. By using the same $N$, $s_n$ will satisfy CC$(s_n,s)$ .
		\item To prove the second clause, take the sequence $t_n=\frac{1}{n}$ where $n\in\mathbb{N}$. This sequence satisfies CC$(t_n,0)$, but it does not satisfy HC$(t_n,s)$ for any $s$ (as any $r\in\Delta$ will satisfy the negation of Definition \ref{def:HC}).
	\end{enumerate}
\end{proof}

\begin{lemma}
	\label{lem:convAbsSeq}
	Let $(s_n)$ be a sequence in $\Rzl$ such that HC($s_n,s$) is hold. Then, HC($\abs{s_n},\abs{s}$) is hold.
\end{lemma}
\begin{proof}
	Let $r>0\in\Rzl$ be given. Then this means that there exists $N\in\mathbb{N}$ such that $\forall m>N$, $\abs{s_m-s}<r$. Therefore, we also have
	\begin{align*}
		\forall m>N\textnormal{, }\abs{\abs{s_m}-\abs{s}}\leq\abs{s_m-s}<r.
	\end{align*}
	Hence, HC($\abs{s_n},\abs{s}$) is true.
\end{proof}


\begin{theorem}
	Let $X\subset\Rzl$ and $f:X\rightarrow\Rzl$. Then $f$ is $\einf$-continuous at $x_0\in X$ iff for any sequence $x_n$ in $X$ that satisfies HC$(x_n,x_0)$, the sequence $f(x_n)$ satisfies HC$(f(x_n),f(x_0))$.
\end{theorem}
\begin{proof}
	Suppose that $f$ is $\einf$-continuous at $x_0$ and let the sequence $x_n$ be defined in $X$ and that $x_n$ hyper converges to $x_0$. Now let $\einf>0$ be given. Then from Theorem \ref{thm:etopcontIffED}, there exists $\einf_2>0\in\Rzl$ such that
	\begin{center}
		if $x\in X$ and $\abs{x-x_0}<\einf_2$, then $\abs{f(x)-f(x_0)}<\einf$.
	\end{center}
	Now since $x_n$ hyper converges to $x_0$, then there exists $N\in\mathbb{N}$ such that $\forall n\geq N$ $\abs{x_n-x_0}<\einf_2$. Thus we have
	\begin{align*}
	\forall n\geq N \textnormal{ } \abs{f(x_n)-f(x_0)}<\einf.
	\end{align*}
	and so the sequence $f(x_n)$ hyper converges to $f(x_0)$.\newline
	For the converse, we will prove the contrapositive. Suppose that $f$ is not $\einf$-continuous at $x_0$. Then it means that there exists $\einf_0>0\in\Rzl$ such that for all $\einf_2>0\in\Rzl$, there exists $x\in X$ such that $\abs{x-x_0}<\einf_2$ but $\abs{f(x)-f(x_0)}>\einf_0$. In particular, for all $n\in\mathbb{N}$, there exists $x_n\in X$ such that $\abs{x_n-x_0}<\einf_2$ and $\abs{f(x_n)-f(x_0)}>\einf_0$. Thus $x_n$ is a sequence in $X$ that hyper converges to $x_0$, but the sequence $f(x_n)$ does not hyper converge to $f(x_0)$.
\end{proof}

\begin{definition}
	Let $s_n$ be a sequence in $\Rzl$. Then we say that $s_n$ is a hyper-Cauchy sequence iff $\forall\einf\in\Rzl$, $\exists N\in\mathbb{N}$ such that
	\begin{align*}
		\forall l,m\geq N\textnormal{ }\abs{s_l-s_m}<\einf.
	\end{align*}
\end{definition}

\begin{theorem}
	Every hyper convergent sequence in $\Rzl$ is a hyper-Cauchy sequence.
\end{theorem}
\begin{proof}
	Let $s_n$ be a sequence in $\Rzl$ that satisfies HC$(s_n,s)$. We want to show that $s_n$ is hyper-Cauchy. Let $\einf\in\Rzl$ be given. Then there exists $N\in\mathbb{N}$ such that $\forall n>N$, $\abs{s_n-s}<\frac{\einf}{2}$. Then for all $l,m>N$, we have
	\begin{align*}
		\abs{s_l-s_m}=\abs{s_l-s-(s_m-s)}\leq\abs{s_l-s}+\abs{s_m-s}<\frac{\einf}{2}+\frac{\einf}{2}=\einf
	\end{align*}
	and so $s_n$ is hyper-Cauchy.
\end{proof}

\begin{conjecture} 
	The set $\Rzl$ is hyper-Cauchy complete with respect to the $\einf$-topology.
\end{conjecture}

\begin{lemma}
	Let $s_n$ be a sequence in $\Rzl$ whose members are just real numbers -- that is, for all $s\in s_n$, $\texttt{Nst}_\epsilon(s)=\texttt{Nst}_\omega(s)=\emptyset$. Then $s_n$ is hyper-Cauchy if and only if there exists $N\in\mathbb{N}$ such that $s_m=s_N$ for all $m\geq N$.
\end{lemma}
\begin{proof}
	Let $s_n$ be a hyper-Cauchy sequence in $\Rzl$ whose members are real numbers. Then there exists $N\in\mathbb{N}$ such that
	\begin{align}
		\label{eq:5.13}
		\abs{s_m-s_l}<\epsilon \textnormal{ for all } m,l\geq N.
	\end{align}
	Since $s_n$ is a sequence of real numbers, we obtain from Inequality \ref{eq:5.13} that for all $m,l\geq N$, $\abs{s_m-s_l}=0$ and so $s_m=s_N$ for all $m\geq N$.\newline
	Conversely, let $s_n$ be a sequence in $\Rzl$ whose members are real numbers and assume that there exists $N\in\mathbb{N}$ such that $s_m=s_N$ for all $m\geq N$. Now let $\einf>0$ be given. We have that for all $l,m\geq N$, $\abs{s_m-s_l}=0<\einf$ and so $s_n$ is hyper-Cauchy.
\end{proof}

Another possible way to define convergence in our set is through the concept of $\ell^\infty$ as follows:

\begin{definition}[$\Rzl$-Convergence]
	\label{def:linftyConv}
	Suppose that $s_n$ is a sequence where every member of it is another sequence itself, i.e.
	\begin{center}
		$s_n=(s_n)_1,(s_n)_2,(s_n)_3,\dots,(s_n)_i,\dots$.
	\end{center}
	Then, $s_n$ converges to $\mathbf{s}$ iff $\forall m\in\mathbb{N}$, $\exists N$ such that 
	\begin{center}
		$\forall n\geq N$, $\forall i$ $\abs{(s_n)_i-s_i}<\frac{1}{m}$.
	\end{center}
	We write RC$(s_n,s)$ to denote that a sequence $s_n$ is $\Rzl$-convergent to $s$.	
\end{definition}

\begin{example}
	The sequence $s_n=\langle\widehat{\frac{1}{n}},0,0,\dots\rangle$ is $\Rzl$-convergent to $\textbf{0}$.
\end{example}

The next interesting question is which of the three definitions above can be used to define convergence in $\Rzl$? Unfortunately, neither of them is adequate to serve as \textit{the} definition of convergence in our set. The three examples below demonstrate the reason. The first example shows that when Classical Convergence is adopted in $\Rzl$, convergence is no longer unique. While the second one shows how adopting \Cref{def:HC} gave something unexpected occurs in our set, the last example shows why $\Rzl$-convergence is not adequate.

\begin{example}
	\label{example:problemStdConv}
	Suppose that $s_n$ is a sequence defined by:
	\begin{center}
		$s_n=\langle \widehat{0},n,0,\dots\rangle$.
	\end{center}
	Then by using \Cref{def:CC} above and the fact that any infinitesimals are less than any rational numbers, $s_n$ classically converges to $100\epsilon$, $200\epsilon$, $300\epsilon$, and so on. In other words, the sequence $s_n$ satisfies (CC$(s_n,100\epsilon)$), (CC$(s_n,200\epsilon)$), (CC$(s_n,300\epsilon)$), and so on.
\end{example}

\begin{example}
	\label{example:hyperconvFails}
	Using Definition \ref{def:HC}, the sequence $s_n=\langle\widehat{\frac{1}{n}},0,0,\dots\rangle$ does not converge in the usual sense to 0, i.e. $s_n$ does not satisfy HC$(s_n,\textbf{0})$. Taking $r=\epsilon=\langle\widehat{0},1,0,\dots\rangle$ and $n=N+1$ will show this.
\end{example}

\begin{example}
	The sequence $s_n=\epsilon^n$ does not $\Rzl$-converge to $\textbf{0}$, as it should do intuitively.
\end{example}

Thus, this leaves us with the three definitions of convergence used in $\Rzl$. There is no one definition of convergence in our set. This is not necessarily a bad thing, it simply means that our notion of convergence will differ from that of classical analysis.

Note that our attempts to have a proper notion of continuity and convergence in $\Rzl$ can be used in the area of reverse mathematics. From what we have done here, it can help us to gain a better understanding about some \textit{necessary condition}, for example, for a function $f$ to be continuous or for a sequence to be convergent.

\section{Some Notes on The Computability in $\Rzl$}
\label{sec:ComputabilityInRZ<}

A computable function is a function $f$ which could, in principle, be calculated using a mechanical calculation tool and given a finite amount of time. In the language of computer science, we would say that there is an algorithm computing the function. A computable real number is, in essence, a number whose approximations are given by a computable function.

The notion of a function $\mathbb{N} \to \mathbb{N}$ being computable is well understood. In fact, all definitions that so far capturing this idea (such as Turing Machines, Markov Algorithms, Lambda Calculus, the (partial) recursive functions, and many more) have all led to the same class of functions. This, in turn, has led to the so called Church-Markov-Turing thesis, which says that this class is exactly what computable intuitively means. Given computable pairing functions also, immediately, lead to a notion of computability for other function types such as $\mathbb{N}^k \to \mathbb{N}^m$, $\mathbb{N} \to \mathbb{Z}$ or $\mathbb{N} \to \mathbb{Q}$. If we see a real number as a sequence of rational approximation, we also get a definition of a computable real number.

However, we have to be a bit careful. There are many equivalent formulations for when a real number $r$ is computable, that work well in practice. This happens such as when 

\begin{itemize}
	\item there is a finite machine that computes a quickly converging\footnote{That is with a fixed modulus of Cauchyness.} Cauchy sequence that converges to $r$, or
	\item it can be approximated by some computable function $f:\mathbb{N}\rightarrow\mathbb{Z}$ such that: given any positive integer $n$, the function produces an integer $f(n)$ such that
	\begin{center}
		$\frac{f(n)-1}{n}\leq a \leq \frac{f(n)+1}{n}$.
	\end{center}
\end{itemize}

\noindent We denote the set of all computable real numbers by $\R_c$. It is well known (and also well studied) that many real numbers, such as $\pi$ or $e$, are computable. However, not every real number is computable.

One possibility that does not turn out to be useful is to write down a real number by using its decimal representation.\footnote{Consider a number $r$ such that there is an algorithm whose input is $n$, and it will give the $n^\textnormal{th}$-digit of $r$'s decimal representation.} The set of all real numbers  that have a computable decimal representation is denoted by $\R_d$.

\begin{remark}
	Although the set $\R_d$ is closed under the usual arithmetic operations, we have to be careful of what it really means. Take, for example, addition. We know that if $x$ and $y$ are in $\R_d$, then $x+y$ is also in $\R_d$. However, it does \emph{not} mean that the addition operator itself is computable.
\end{remark}

These ideas of computability can be extended to infinitesimals. In $\Rzl$, we define its member to be computable if it satisfies the condition as stated in Definition \ref{def:computablenumberRZ<}.

\begin{definition}
	\label{def:computablenumberRZ<}
	A number $\mathbf{z}\in\Rzl$ is computable iff there is a computable function $f$ such that $f(n,\cdot)$ are computable numbers and
	\begin{center}
		$\mathbf{z}=\langle f(1,\cdot),f(2,\cdot),\dots,\widehat{f(l,\cdot)},\dots\rangle$	
	\end{center}
	where $l=f(0,0)$ denotes the index where the $\texttt{St}(\mathbf{z})$ is. We denote the set of all computable members of $\Rzl$ by $\Rzlc$.
\end{definition}

In this section, we showed that the standard arithmetic operations (functions) in $\Rzl$ are computable (provided that the domain and codomain of those functions are (in) $\Rzlc$). This was done by explicitly showing the program for each one of them. We actually uses a concrete implementation of these ideas in the programming language Python, whose syntax should be intuitively understandable even by those not familiar with it. There is also no need to show that our programs are correct, since they are so short that such a proof would be trivial.

Assuming that we already had a working implementation of $\R_c$, our class $\Rzlc$ could be implemented as in Listing \ref{mycode}. There we defined the members of our set $\Rzlc$ (basically just a container for the index $l$ as in \Cref{def:computablenumberRZ<} and the sequence of digits) and how their string representation would look like.

\ContinueLineNumber

\begin{lstlisting}[language=Python, caption=How to define the members of $\Rzlc$., label=mycode, escapeinside={(*}{*)}, numbers=none]
class infreal:
	def __init__(self, digits, k=0):
		self.k = k
		self.digits = digits
	def __repr__(self):
		if self.k == 0:
			return "^" + ", ".join([str(self.digits(i)) for i in range(self.k,self.k+7)]) + ", ..."
		else: 
			return ", ".join([str(self.digits(i)) for i in range(self.k)]) + ", ^" + ", ".join([str(self.digits(i)) for i in range(self.k,self.k+7)]) + ", ..."
	def __getitem__(self, key): return self.digits(key)
\end{lstlisting}

\begin{example}
	\label{ex:6.4}
	Suppose that we want to write the number $\mathbf{1}=\langle\widehat{1},0,0,\dots\rangle$. Then by writing
	\begin{center}
		\pycode{One=infreal(lambda n:one if n==0 else zero, 0)}
	\end{center}
	where \texttt{\textcolor{red}{zero}} and \texttt{\textcolor{red}{one}} are the real numbers $0$ and $1$, respectively, we just created the number $\mathbf{1}$ in our system. The second argument of the function \texttt{\textcolor{red}{infreal}} is just to give how many digits we want to have before the real part of our number (the number with a hat). Its input and output will look like as follows:
	\begin{lstlisting}[language=Python, numbers=none]
	>>> zero = real(0)
	>>> one = real(1)
	>>> One = infreal(lambda n: one if n==0 else zero, 0)
	>>> One
			^1, 0, 0, 0, 0, 0, 0, ...\end{lstlisting}
	Furthermore, we will also be able to know what is its $n^{\textnormal{th}}$ digit for any $n\in\mathbb{N}$. See the code below:
	\begin{lstlisting}[language=Python, numbers=none]
		>>> One
				^1, 0, 0, 0, 0, 0, 0, ...
		>>> One[0]
				1
		>>> One[-56]
				0
		>>> One[2454]
				0\end{lstlisting}
\end{example}

\begin{example}
	Similar to \Cref{ex:6.4}, the numbers $\mathbf{\epsilon}$ and $\mathbf{\omega}$ could also be defined in our system.
	\begin{lstlisting}[language=Python, numbers=none]
		>>> Epsilon = infreal(lambda n: one if n == 1 else zero, 0)
		>>> Epsilon
				^0, 1, 0, 0, 0, 0, 0, ...
		>>> Omega = infreal(lambda n: one if n == 0 else zero, 1)
		>>> Omega
				1, ^0, 0, 0, 0, 0, 0, 0, ... \end{lstlisting}
	Also, we will be able to have exotic numbers such as $\mathbf{\me}+\mathbf{2\me\epsilon}+\mathbf{3\me\epsilon}^2+\mathbf{4\me\epsilon}^3+\dots$ and its code will be as follows:
	\begin{lstlisting}[language=Python, numbers=none]
		>>> e = exp(rational(1,1))
		>>> Funny = infreal(lambda n: real(rational(n + 1, 1)) * e if n > -1 else zero, 0)
		>>> Funny
				^2.71828, 5.43656, 8.15485, 10.8731, 13.5914, 16.3097, 19.028, ...
		>>> Funny[43532]
				118335
		>>> Funny[-12964]
				0\end{lstlisting}
\end{example}

\Cref{thm:addIsComputable}-\ref{thm:mulIsComputable} show that addition, subtraction, and multiplication in $\Rzlc$ are computable.

\begin{theorem}
	\label{thm:addIsComputable}
	Suppose that we have $x,y\in\Rzlc$. Then the function $\widehat{+}_c$ defined by
	\begin{eqnarray*}
		\widehat{+}_c: & \Rzlc & \rightarrow\Rzlc\\
		& (\mathbf{x},\mathbf{y})& \mapsto \mathbf{x}\widehat{+}\mathbf{y}
	\end{eqnarray*}
	is computable.
\end{theorem}
\begin{proof}
	The following code shows that the function $\widehat{+}_c$ defined above is computable.
	\begin{lstlisting}[language=Python, numbers=none]
		def __add__(self, other):
			k = max(self.k, other.k)
			return infreal(lambda n: self.digits(n - (k-self.k)) + other.digits(n - (k-other.k)), k)\end{lstlisting}
\end{proof}

\begin{example}
	Suppose that we want to add $\mathbf{\epsilon}$, $\mathbf{\omega}$, and $\mathbf{1}$. Then we will have:
	\begin{lstlisting}[language=Python, numbers=none]
	>>> Omega
	1, ^0, 0, 0, 0, 0, 0, 0, ...
	>>>	Epsilon
	^0, 1, 0, 0, 0, 0, 0, ...
	>>> Epsilon + Omega + One
	1, ^1, 1, 0, 0, 0, 0, 0, ...\end{lstlisting}
\end{example}

\begin{theorem}
	\label{thm:minusIsComputable}
	Suppose that we have $x,y\in\Rzlc$. Then the function $\widehat{-}_c$ defined by
	\begin{eqnarray*}
		\widehat{-}_c: & \Rzlc & \rightarrow\Rzlc\\
		& (\mathbf{x},\mathbf{y})& \mapsto \mathbf{x}\widehat{-}\mathbf{y}
	\end{eqnarray*}
	is computable.
\end{theorem}
\begin{proof}
	The following code shows that the function $\widehat{-}_c$ defined above is computable.
	\begin{lstlisting}[language=Python,escapeinside={(*}{*)}]
	def __neg__(self): return infreal(lambda n: -self[n], self.k) (*\label{line8}*)
 	def __sub__(self, other): return (self + (-other))\end{lstlisting}
 	The definition in line \ref{line8} shows that the additive inverse function is computable.
\end{proof}

\begin{example}
	Suppose that we want to add $\mathbf{\epsilon}\widehat{-}\mathbf{\omega}$ to $\mathbf{1}$. Then we will have:
	\begin{lstlisting}[language=Python, numbers=none]
		>>> Epsilon - Omega + One
				-1, ^1, 1, 0, 0, 0, 0, 0, ...\end{lstlisting}
\end{example}

\begin{theorem}
	\label{thm:mulIsComputable}
	Suppose that we have $x,y\in\Rzlc$. Then the function $\widehat{\times}_c$ defined by
	\begin{eqnarray*}
		\widehat{\times}_c: & \Rzlc & \rightarrow\Rzlc\\
		& (\mathbf{x},\mathbf{y})& \mapsto \mathbf{x}\widehat{\times}\mathbf{y}
	\end{eqnarray*}
	is computable.
\end{theorem}
\begin{proof}
	The following code shows that the function $\widehat{\times}_c$ defined above is computable.
	\begin{lstlisting}[language=Python,escapeinside={(*}{*)}, numbers=none]
	def __mul__(self, other):
		k = self.k + other.k
		def digits(n):
			if n < 0:
				return zero
			else:
				return reduce((lambda x,y:x+y), [self.digits(i) * other.digits(n - i) for i in range(n + 1)])
		return infreal(digits, k)  \end{lstlisting}
\end{proof}

\begin{example}
	Suppose that we want to add $\mathbf{\epsilon}\widehat{\times}\mathbf{\omega}$ to $\mathbf{1}$ and also $-\mathbf{\epsilon}^2$ to $\mathbf{1}$. Then we will have:
	\begin{lstlisting}[language=Python, numbers=none]
	>>> Epsilon * Omega + One
			0, ^2, 0, 0, 0, 0, 0, 0, ...
	>>> Epsilon * -Epsilon + One
			^1, 0, -1, 0, 0, 0, 0, ...\end{lstlisting}
\end{example}

\subsection{Some Remarks on Non-Computability in $\Rzl$}

\begin{remark}
	Even though division on $\R_c$ is computable (assuming the input does not equal $0$), the same can not be said of $\Rzlc$.
\end{remark}

\begin{remark}
	Suppose that we have a number $\mathbf{x}\in\Rzl$. Then without any further information, the process of finding $\mathbf{x}^{-1}$ (the multiplicative inverse of $\mathbf{x}$) is not computable. One extra information needed to make it computable in $\Rzl$ is how many digits we want to have in $\mathbf{x}^{-1}$, which of course will affect the accuracy of our result. More precisely, it is known that it is not possible to give an algorithm that, a given number $a \in \R_c$, decides whether $a = 0$ or $\lnot a =0$. So let $a \in \R_c$ and consider $\mathbf{z} = a + \mathbf{\epsilon}$. We have $ \mathbf{\epsilon} \neq 0$. If $a = 0$ then $\mathbf{z}^{-1} = \mathbf{\omega}$. If $a \neq 0$ then $z^{-1} < \mathbf{\omega}$. Thus by checking whether the $\omega$-part of $z^{-1}$ is less than $1$ or greater than $0$, we would be able to decide whether $a = 0$ or $\lnot a =0$.
\end{remark}

\begin{remark}
Similarly surprising, we can show that the absolute value function, which is computable for $\R_c$, is not computable for $\Rzlc$. Here the absolute value function is the function 
\[ \left| \mathbf{z} \right| = \begin{cases}
	\mathbf{z} & \text{if } \mathbf{z} \geq 0 \\ -\mathbf{z} & \text{if } \mathbf{z} < 0 
\end{cases} \]
Similar to the above, take $a \in \R_c$ and consider $z = |a| - \mathbf{\epsilon}$. If $a =0$ then $|z| =- |a| + \mathbf{\epsilon}$. If $a \neq 0$ then $|z| = |a| - \mathbf{\epsilon}$. Thus by checking the $\epsilon$-part of $|z|$, we would be able to decide whether $a = 0$ or $\lnot a =0$.
\end{remark}

This also leads to comparison between numbers not being computable. Now this is also the case in $\R_c$. However, for numbers $x,y \in \R_c $ such that $x \neq y$, we can decide whether $x<y$ or $x > y$. This does not extend to $\Rzlc$: 

\begin{remark}
Comparison among the members in $\Rzlc$ is not computable. 
Again, let $a \in \R_c$ and consider $\mathbf{x} = |a| + \mathbf{\epsilon}$ and $\mathbf{y} = \textbf{2}|a|$. If $a = 0$ then $\mathbf{x} > \mathbf{y}$, and if $a \neq 0$ then $\mathbf{x} <\mathbf{y}$. Thus, once again we would be able to decide whether $a = 0$ or $\lnot a =0$.
\end{remark}

\section{Conclusion and Suggestions for Further Research}
\label{sec:Conclusion}

In this article we treated the nonstandard real numbers in the spirit of the `Chunk and Permeate' approach. The sets $\R$ and $\mathbb{^*R}$ were thrown together (that means: combining their languages and axioms) and arising inconsistency issues were dealt with using paraconsistent reasoning strategy. In this case --- and this is one of the interesting novelties --- the procedure was made explicit by introducing new sets $\mathbb{\widehat{R}}$ and in turn $\Rzl$, the separate `chunks'. In a sense, we transferred the `Chunk and Permeate' approach from the theoretical level to the (explicit) model level. Whereas the impact on infinitesimal mathematics was only sketched in \cite{brown2004chunk}, it was worked out in detail here, using the sets $\mathbb{\widehat{R}}$ and $\Rzl$. After introducing the theoretical background, we constructed the new model of nonstandard analysis in detail. The remaining part of this paper lies in an extensive discussion of topological, applied (in the sense of calculus), and computability issues of the obtained model. A side result of the constructed set $\Rzl$ was a direct consistency proof of the Grossone theory, see \cite{lolli2015metamathematical}.

On the \textit{topological} aspect in \Cref{sec:TopologyOnRZ<}, we introduced some new notions of metrics, balls, open sets, and etc in $\Rzl$ together with their properties. On the \textit{applied} aspect in \Cref{sec:CalculusOnRZ<}, some new concepts on the calculus in $\Rzl$ were discussed, e.g. derivative (we successfully developed a permeability relation such that the derivative function in $\Rzl$ can be permeated to $\mathbb{R}$), continuity, and convergence. For the two last issues, some new notions were introduced in this article. First of all, we discussed the three possible notions of continuity that can be applied to either $\mathbb{R}$ or $\Rzl$. We also determined how they relate to each other in their respective model. While doing that, we discovered a new kind of fractals --- infinitesimal fractals. After analysing three possible notions of continuity, we decided that the best notion that can be used in our setting $\Rzl$ is the $\epsilon$-$\delta$ definition and by doing that, we do not only preserve much of the spirit of classical analysis but also retain the intuition of infinitesimals. After establishing our position, we introduced a more detailed notion of continuity which is called ($k,n$)-continuity (as can be seen in Definition \ref{def:k-ncontinuity}). We explored how this new notion of continuity behaves, e.g. what happens with the composition of two continuous functions and also how this notion behaves under multiplication. It is worth pointing out here that this new notion of continuity is a much more fine-grained notion than the classical continuity. Last but not least, we showed that the set $\Rzl$ has nice \textit{computability} features. We succeeded in building a program, in Python, to show that we can have a computable number $\Rzl$. The set of all these computable numbers is denoted by $\Rzlc$.  We also showed some interesting remarks regarding this computability issue.

In term of further research, we indicated some possible areas of further development as follows. \textit{First}, one could try to do infinitesimal analysis using the relevant  logic \textbf{R}. The comparison between the results (perhaps) gotten in $\textbf{R}$ and the one described here might be interesting, especially in term of usefulness and simplicity. \textit{Secondly}, regarding the `transfer principle', our intuition says that it is equivalent to the notion of permeability in the Cchunk \& Permeate strategy. One could try to formally prove it, or disprove it. \textit{Thirdly}, in term of computability issue, using the calculus on $\Rzl$, one could try to formulate the necessary and sufficient conditions for the derivatives of functions, for example, on a computer to exist. And perhaps, showing also how to find these derivatives whenever they exist. This, of course, can also be applied to the other notions. \textit{Fourthly}, as been said in the previous sections, some results described in this article could help us to gain a better understanding in another area of research (the two that were mentioned in \Cref{sec:CalculusOnRZ<} are reverse mathematics and quantum physics). One could try to work out the details on this.

In general, with the new consistent sets created in this work, new opportunities awaits mathematicians. One of the joys of mathematics is to explore a world which has no physical substance, and yet is everywhere in every aspect of our lives. Infinities and infinitesimals offer ways to explore hitherto unseen aspects of our world and our universe, by giving us the vision to see the greatest and smallest aspects of life. Even a na\"ive set, when it demonstrates harmony, offer another dimension of even clearer precision. In a wide sense, the work on this article can also be seen as a contribution to bridge (the antipodes) constructive analysis and nonstandard analysis. This problem has been extensively (and intensively) discussed in the past few years (see for example \cite{sanders2013connection,sanders2015effective,sanders2017nonstandard,bournez2018cheap,sanders2018gandy,normann2019computability}).

\renewcommand{\abstractname}{Acknowledgements}
\begin{abstract}
	The first author received financial support from Indonesia Endowment Fund for Education that enables the research of this article. We would also like to thank Professor Elem\'er Rosinger and Dr.~Josef Berger for their invaluable inputs.
\end{abstract}

%
%

\bibliographystyle{spmpsci}      
\bibliography{references.bib}   

\end{document}